\documentclass[a4paper, 12 pt]{article}
\usepackage[utf8]{inputenc}
\usepackage{amssymb, latexsym, amsmath, amsthm, mathrsfs}
\usepackage{mathpazo,stmaryrd}

\usepackage{enumitem}
\usepackage{multirow}
\usepackage{hyperref}
\usepackage{soul}

\usepackage{xcolor}
\usepackage{tikz}
\usetikzlibrary{arrows}

\usepackage{authblk}

\newtheorem{thm}{Theorem}[section]
\newtheorem{definition}{Definition}[section]
\newtheorem{lemma}{Lemma}[section]
\newtheorem{cor}{Corollary}[section]
\newtheorem{obs}{Observation}[section]

\newcommand{\ogeq}{\ogreaterthan}
\newcommand{\ssq}{\subseteq}

\title{The (Relevant) Logic of Scientific Discovery}
\author[1]{Timothy Childers}
	
\author[1]{Ondrej Majer} 

\author[2] {Peter Milne} 

\affil[1]{\small Czech Academy of Sciences, Institute of Philosophy, Prague\\
	childers@flu.cas.cz, majer@flu.cas.cz}
\affil[2]{Dept. of Philosophy, University of Stirling, 
	Stirling\\
	peter.milne@stir.ac.uk}
\begin{document}
\date{}
\maketitle

\begin{center}
\Large Part I: Semantics
\end{center}

\section{Introduction}
This paper presents a novel, thorough-going interpretation of some relevant logics. The interpretation employs an idealized modelling of the search for regularities in scientific inquiry. Laboratories (research teams) set up, carry out, and assess experiments, and thereby arrive at regularities or their absence. The modelling captures the distinction between confirmation and refutation, a distinction which motivates a rejection of contraposition.

We found our interpretation on relevant logics built over Dunn--Belnap four-valued semantics for negation, conjunction and disjunction. We employ two accessibility relations: one describing confirmation of a regularity, the second disconfirmation. It is this fine-grained approach that necessitates the inadmissibility of contraposition.

While novel our interpretation is not without anticipations of some of the details. There have been many nods to interpretations similar to ours---for example, in \cite{Dunn1976,Belnap1977,Belnap1977a,Belnap1992}---and formally our account falls under what Richard Sylvan called ``the American plan completed'' \cite{Routley1984}. However, while Sylvan also employs two accessibility relations, his system is significantly more complicated, in the main because he sets out to preserve contraposition. There are only a few other examples of a four-valued logic explicated in terms of two accessibility relations. Recently Takuro Onishi \cite{Onishi2016} has employed the second relation to provide a semantics for co-implication.

Finally, data is messy. We begin to tackle this in the second part of this paper by adding probabilities to the logical framework. We begin with probabilities over Dunn-Belnap logics, and show how they can be interpreted as relative frequencies and betting quotients. We consider a number of updating strategies, including updating for groups of labs. 



\section{Motivation}\label{motive}

We begin with laboratories, or, more broadly, teams or groups of researchers, that learn, and then share what they've learned with other laboratories or teams. To keep things brief, we'll talk about laboratories but the model has wider application. The laboratories' endeavours can be divided into four stages. First, and obviously, there is a planning stage. This is followed by the initiating event, for example, the sending out of a signal, the administration of a treatment regime, the delivery of questionnaires. Then there's a follow-up event, possibly distant in space and/or time, for example, the receipt of the signal, the collection of data on the effect of the treatment regime, the gathering in of completed questionnaires. Finally, the previous stages are evaluated. All four stages may be carried out by the same individual or team, or there may be some division of labour, perhaps widely distributed so that, say, multiple signals are sent out and received at various sites, groups of patients with different medical histories are subject to the same treatment regime, questionnaires are distributed to a variety of sample populations in diverse locales.

We can reduce the four stages to three by ignoring the planning stage. This seems reasonable: while, on the one hand, an experiment may be proposed and only decades later carried out, the team evaluating the results has to know what the initiating and follow-up teams report, how they are co\"{o}rdinated, and how to evaluate those reports; it does not need to know what prompted the initiating and follow-up teams to do what they did.\footnote{Teams need not be identified with particular sets of individuals. Certainly, they should not be in long-term longitudinal studies such as the UK's on-going National Child Development Study which focuses on over 17,000 people born in a single week in 1958.}

We can then describe our labs at the various stages with a ternary relation. We will say that $R_{1}xyz$ iff \\
\begin{quote}
$x$ collates and evaluates the data obtained from a test set-up initiated by $y$ and followed up by $z$.
\end{quote}

In searching for regularities the laboratories look for perfect correlations between the initiating events of the experiment and the outcomes of that experiment. The labs seek to determine if a given experimental set-up always leads to a particular outcome, and never leads to an alternative outcome. This is naturally represented by a conditional:

\begin{quote}
the conditional $A \rightarrow B$ is true from the point of view of a lab $s$, if, for every pair of labs $\langle t, u \rangle$ in $s$'s purview testing the regularity, $t$ initiating and $u$ following up, when lab $t$ reports that $A$ is true,  lab $u$ reports that $B$ is  true.
\end{quote}

\noindent In our framework, the search for regularities supervenes on tests for true conditionals.\footnote{This, we should emphasise, is our take on scientific \emph{practice}. It can be thought of as parasitic upon the the widespread but also widely derided reading of scientific laws and regularities as quantified conditionals of which the paradigm in the literature seems to be `All ravens are black'. In \cite{Friend2017} Toby Friend `offer[s] some argument for the view that laws do indeed have a quantified conditional form.' It's important to note that Friend's argument `for the conditional feature of the schema of laws aims to show that a conditional is implicit in our very understanding of laws, even when they do not appear to explicitly have that form' \cite[p.\ 127]{Friend2017}. A claim such as that chocolate consumption is inversely associated with prevalent
coronary heart disease (which may be the case \cite{Djousse2011}) can be turned into a generalization---{}a ``law''---{}concerning cross-sectional studies.} 

But we need to say a little more about what `true' amounts to in this setting. Lab reports may be positive and negative of course: that the experiment was set up in the proper way or not, or that an outcome was observed or not. But they can also be null: experiments might not be set up at all, or outcomes might not be observed at all (equipment breaks down, subjects aren't compliant, etc.). And the same labs can also report contradictory findings (equipment is on the blink, a subject is schizophrenic).\footnote{``In some cases, and perhaps in all, observations of two sufficiently similar entities are inconsistent when the same comparison is repeated several times.'' \cite[p.\ 3]{Krantz1971}} This leads to four possible values, \textit{true}, \textit{false}, \textit{neither true nor false}, \textit{both true and false}.\footnote{Readers of a nervous disposition may substitute `reported true' or `adjudged true' or `accepted as true' for our use of `true' and likewise, \textit{mutatis mutandis}, for our use of `false'. \textit{Cf.} Nuel Belnap's `told True' and `told False' in \cite{Belnap1977} and Michael Dunn's disclaimer:
\begin{quote}
Do not get me wrong---{}I am not claiming that there are sentences which are in fact both true and false. I am merely pointing out that there are plenty of situations where we suppose, assert, believe, etc., contradictory sentences to be true, and we therefore need a semantics which expresses the truth conditions of contradictions in terms of the truth values that the ingredient sentences would have to take for the contradictions to be true. \cite[p.\ 157]{Dunn1976}
\end{quote}}

Of special significance for our story  will be a subset of $S$ of well-behaved labs, picked out by simple criteria. Firstly, labs work on multiple projects. A lab that has oversight where one set of experiments is concerned may be supplying data to another testing other regularities. Experimentation is organized and regulated activity. For pairs of labs (or research groups or experimental teams), one initiating the other completing diverse experiments, possibly overseen by diverse others, the completing lab has to know what the initiating lab has done---{}for only so will it know it has work to do, data to collect and communicate to the relevant overseeing lab. Thus when $x$ is a well-behaved lab making a good job of its co\"{o}rdinating role, and $R_{1}xyz$, either by making the same determinations itself or by having them reported to it, $z$ must have made the same acceptances as true and false as $y$. We denote the fact that, for some well-behaved lab $x$, $R_{1}xyz$ by $y \le z$. In effect, this is an information ordering on labs: $z$ has all the information $y$ has (and possibly more since $z$ has to complete any experiment $y$ has initiated under $x$'s oversight). So read, we should expect it to be reflexive and transitive---{}a pre-order. Secondly, some regularities---{}\textit{e.g.} anything of the form $A \to A$---{}are laws of logic. Any well-behaved lab should adjudge \emph{those} regularities as true. $L$ designates the (non-empty!) subset of $S$ comprising the well-behaved labs.


All labs conduct "equipment checks" -- calibrating instruments, cleaning test tubes, approving questionnaires. We assume that a well-behaved lab has oversight of these activities. That is

\begin{enumerate}[label=(\alph*)]
\item for all $x \in S$, $(\exists u \in L)R_{1}uxx$.
\end{enumerate}


Suppose that $u$ is a well-behaved lab, that $R_{1}uxy$ and that $y \leq z$. Then the $z$ team, having made all the determinations, whether off its own bat or by having them reported to it, that $y$ has, can stand in for $y$ in supplying $u$ with the data it needs about the test initiated by $x$. That is
\begin{enumerate}[label=(\alph*)]
\addtocounter{enumi}{1}
\item if $u \in L$, $R_{1}uxy$ and $y \leq z$ then $R_{1}uxz$.
\end{enumerate}

Somewhat analogously, if $R_{1}xyz$ and $w \leq x$ then $w$ having made no determinations that $x$ has not, knows nothing that stands against $y$'s and $z$'s reports that $x$ does not and hence can use their determinations in its own evaluations of regularities. That is
\begin{enumerate}[label=(\alph*)]
\addtocounter{enumi}{2}
\item $R_{1}xyz$ and $w \leq x$ then $R_{1}wyz$.
\end{enumerate}

We'll take it that, to sloganise, ``good practice breeds good practice'', that if $u$ and $v$ are both well-behaved and $R_{1}uvx$ then $u$ and $v$'s good habits rub off/are imposed on $x$. That is
\begin{enumerate}[label=(\alph*)]
\addtocounter{enumi}{3}
\item if $u, v \in L$ and $R_{1}uvx$ then $x \in L$.
\end{enumerate}
It follows already from what was said above that $x$ must acknowledge all regularities that are logical truths.

These are our basic modelling assumptions. Other constraints may be adopted. For example, we might make explicit that the division of labour that sees $y$ initiating and $z$ completing the test of a regularity under $x$'s oversight is accidental, not a necessary feature of the experiment. If $R_{1}xyz$, a single research group $w$ could play the $y$ role --- $R_{1}xwz$ --- \emph{and} play the $z$ role --- $R_{1}xyw$. That is
\begin{enumerate}[label=(\alph*)]
\addtocounter{enumi}{4}
\item if $R_{1}xyz$ then $(\exists w \in S)[ R_{1}xwz$ and $R_{1}xyw ]$.
\end{enumerate}

We trust that our basic assumptions (a)--(d) are plausible descriptions of how no doubt idealised, well-organised, co\"{o}rdinated research activity goes. It would, of course, be disingenuous of us to claim that we have fashioned them without \emph{some} thought as to the behaviour of one of the ternary relations in Routley's ``American plan'' semantics for relevance logic but we do think that they capture an aspect of (idealised) scientific practice---{}far from the full story, of course, but the basics.

The obvious next step is to say that the overseeing lab $x$ reports the regularity $A \to B$ to be true when, for each pair $\langle y, z \rangle$, $y$ initiating, $z$ following up, in its purview, $z$ reports that $B$ is true when $x$ reports that $A$ is true. As far as it goes that's just dandy. When we turn to think of a lab declaring a putative regularity to be false, we need to keep two considerations in mind. First, gathering positive evidence for a hypothesis can be a very different activity from gathering evidence to disconfirm a hypothesis. (This comes to the fore in our later treatment of probability.) The second is that hypothesised regularities are not to be given up too readily. Scientific experiments are, in principle, reproducible\footnote{``Kant was perhaps the first to realise that the objectivity of scientific statements is closely connected with the construction of scientific theories---{}with the use of hypotheses and universal statements. Only when certain events recur in accordance with rules or regularities, as is the case with repeatable experiments, can our observations be tested---{}in principal---{}by anyone. We do not take our own observations quite seriously, or accept them as scientific observations, until we have repeated and tested them.'' \cite[\S{}8]{Popper1959}} and---ideally, at least---only an in principle \emph{reproducible} negative result gets established as counting against the regularity under investigation.\footnote{``We say that a theory is falsified only if we have accepted basic statements that contradict it. This condition is necessary, but not sufficient; for we have seen that non-reproducible single occurrences are of no significance to science. Thus a few stray basic statements contradicting a theory will hardly induce us to reject it as falsified. We shall take it as falsified only if we discover a \emph{reproducible effect} which refutes the theory.'' \cite[\S{}22, emphasis in the original]{Popper1959} See further \cite[\S{}IV]{Kuhn1961} on how scientists react to anomaly and discrepancy.} (In practice, of course, reproducibility is an ideal which cannot always be met, neither in ``hard'' sciences such as cosmology (by the nature of the case) and particle physics (due to the cost of the equipment) nor in many biological and social sciences (again by the nature of the case but also because attempts at strictly reproducing may well violate ethical standards).\footnote{Note too that `replication can even be hazardous. The German scientist Georg Wilhelm Reichmann was fatally electrocuted during an attempt to reproduce Ben Franklin’s famous experiment with lightning' \cite[p.\ 4972]{Fang2010}.}) So it's not enough for a lab to count the regularity $A \to B$ as false---{}as falsified---{} that one pair in its purview have reported $A$ true and $B$ false. The experimental finding has to be reproducible: as we shall formulate it, the negative result has to remain \emph{available}---{}this is our way of discounting Popper's ``stray'' results. For these reasons we employ a second accessibility relation, $R_{2}xyz$, governing refutations. The \textit{availability constraint} takes the form: if $x \leq w$ and $R_{2}xyz$ then $R_{2}wyz$. That the relations $R_{1}$ and $R_{2}$ are largely independent reflects the fact that the decision to count a putative regularity as falsified may be a methodological decision but it isn't purely formal---{}as we might say, provocatively, it isn't a \emph{logical} decision.


In diverse ways, these considerations guide our formal definitions of frames, models, valuations, logical consequence and logical truth---{}to which we now turn.
\



\section{The Logic of Laboratory Reports}\label{LLR}

At this point we have enough of the pieces in place to proceed with formulating our logic. We begin with the definition of a frame, adapted from the Routley--Meyer semantics for relevant logic.\footnote{Textbook accounts are to be found in \cite[Ch.\ 5]{Read1988} and, very briefly, in \cite[Appendix B]{Mares2004}.}

\begin{definition}\label{defFrame}
A \emph{frame} is a quintuple $\mathscr{F} = \langle S, L, \leq, R_1, R_2 \rangle$, where $S$ is a non-empty set, $L$ a subset of $S$, $\leq$ is an ordering on $S$, and $R_1$ and $R_2$ are ternary relations on $S$. We stipulate that $x \leq y$ iff $(\exists u \in L)R_{1}uxy$. \\
\

$R_1$, $R_2$, $L$ and $\le$ have the properties i) - v) and possibly one or more of the other five. In particular for any $x, y, z, w\in S$:
\begin{enumerate}[label=\roman*)]
\item $x \leq x$
\item if $x \leq y$ and $y \leq z$ then $x \leq z$;
\item $L$ is an upwards $\leq$-closed subset of $(S,\leq)$;
\item if $w \leq x$ and $R_{1}xyz$ then $R_{1}wyz$;
\item if $x \leq w$ and $R_{2}xyz$ then $R_{2}wyz$;
\item $R_{1}xxx$;
\item $R_{2}xxx$;
\item if $R_{1}xyz$ then $(\exists w \in S)( R_{1}xyw$ and $R_{1}xwz )$;
\item if $R_{2}xyz$ then $(\exists w \in S)( R_{1}xyw$ and $R_{2}xwz )$;
\item if $R_{2}xyz$ then $(\exists w \in S)( R_{2}xyw$ and $R_{1}xzw)$.
\end{enumerate}
\end{definition}



We take $\mathcal{L}$ to be the set of formulas of a standard propositional language; \textit{i.e.}, formulas of finite length generated from a fixed stock, $At(\mathcal{L})$, of atomic propositions using the connectives $\neg$, $\land$, $\lor$ and $\to$.

Models are defined more or less as usual but with two atomic persistence/heredity contraints.

\begin{definition}\label{defModel}
A \emph{model} $\mathcal{M}$ is a pair $\langle \mathscr{F}, v \rangle$ where $\mathscr{F}$ is a frame and $v$ is a \emph{valuation}, \textit{i.e.} a function from labs and atomic formulas to subsets of the set $\lbrace T, F \rbrace$ of truth-values. That is, $v: (S \times At(\mathcal{L})) \rightarrow \mathscr{P}(\lbrace T, F \rbrace)$. A valuation $v$ satisfies the pair of constraints

\begin{enumerate}[label=\roman*)]
\item if $x \leq y$ and $T \in v(x, p)$ then $T \in v(y, p)$;
\item if $x \leq y$ and $F \in v(x, p)$ then $F \in v(y, p)$.
\end{enumerate}
\end{definition}

We extend $v$ to a function $v_{\mathcal{M}}$ from labs and formulas to subsets of $\lbrace T, F \rbrace$. We evaluate negation, conjunction and disjunction according to this augmented Hasse diagram, familiar from Belnap's ``useful four-valued logic'' $(\cite{Belnap1977,Belnap1977a,Belnap1992})$.

\begin{center}
\begin{tikzpicture}[-,>=stealth',shorten >=1pt,auto,node distance=2.5cm,
  thick,main node/.style={circle,draw,font=\sffamily\normalsize},main node2/.style={circle,fill=gray!50,draw,font=\sffamily\normalsize}]

  \node[main node][label=above left:{$\emptyset$}] (1) {};
  \node[main node2][label={$\lbrace T \rbrace$}] (2) [above right of=1] {};
  \node[main node][label=below:{$\lbrace F \rbrace$}] (3) [below right of=1] {};
  \node[main node2][label=above right:{$\lbrace T, F \rbrace$}] (4) [above right of=3] {};

  \path[every node/.style={font=\sffamily\small}]
    (1) edge (2)
         edge (3)
    (2) edge (4)
    (3) edge (4);
  \draw[dashed, ->]  (3) edge [bend right] node {$\neg$} (2)
    (1) edge [loop left] node {$\neg$} (1)
    (2) edge [bend right] node {$\neg$} (3)
    (4) edge [loop right] node {$\neg$} (4);
\end{tikzpicture}
\end{center}
\noindent In essence we are retaining what may, following \cite[p.\ 385]{Tappolet2000}, be called ``truisms about truth'': \textit{e.g.}, `the truism that a conjunction is true if and only if its conjuncts are true'. Equally we have the truisms that a disjunction is true if and only at least one disjunct is true and that a disjunction is false if and only if both disjuncts are false. Consider the disjunction $A \lor B$ when $A$ is assigned $\emptyset$ and $B$ is assigned $\lbrace T, F \rbrace$: as $B$ is true, $A \lor B$ is true; as $A$ is not false, $A \lor B$ is not at the same time false, and so it is assigned $\lbrace T \rbrace$ as the Hasse diagram indicates.

\begin{definition}\label{defSatisfaction} For every model $\mathcal{M} =\langle \mathscr{F}, v \rangle$ we define the function $v_{\mathcal{M}} : (S \times \mathcal{L}) \rightarrow \mathscr{P}(\lbrace T, F \rbrace)$ as follows (we omit the superscript when it is clear from the context):

\begin{enumerate}[label=\roman*)]
\item \label{atomic} for atomic $p$, $v_{\mathcal{M}}(x, p) = v(x,p)$;

\item\label{not} $T \in v_{\mathcal{M}}(x, \neg{}A) $ iff $F \in v_{\mathcal{M}}(x, A) $, $F \in v_{\mathcal{M}}(x, \neg{}A) $ iff $T \in v_{\mathcal{M}}(x, A) $;

\item\label{and} $T \in v_{\mathcal{M}}(x, A \land B) $ iff $T \in v_{\mathcal{M}}(x, A)$ and $T \in v_{\mathcal{M}}(x, B)$,\\
$F \in v_{\mathcal{M}}(x, A \land B)$ iff $F \in v_{\mathcal{M}}(x, A)$ or $F \in v_{\mathcal{M}}(x, B)$;

\item\label{or} $T \in v_{\mathcal{M}}(x, A \lor B)$ iff $T \in v_{\mathcal{M}}(x, A)$ or $T \in v_{\mathcal{M}}(x, B)$,\\
$F \in v_{\mathcal{M}}(x, A \lor B)$ iff $F \in v_{\mathcal{M}}(x, A)$ and $F \in v_{\mathcal{M}}(x, B)$.

\end{enumerate}

\noindent We modify the standard Routley--Meyer evaluation conditions for the conditional (and adapt to the four-valued setting):

\begin{enumerate}[label=\roman*)]
\setcounter{enumi}{4}
\item\label{conditional} $T \in v_{\mathcal{M}}(x, A \rightarrow B) \mathrm{~iff~} (\forall y, z \in S)[\mathrm{~if~} R_{1}xyz \mathrm{~and~} T \in v_{\mathcal{M}}(y, A) \mathrm{~then~} T \in v_{\mathcal{M}}(z, B)]$;
\item\label{negconditional}$F \in v_{\mathcal{M}}(x, A \rightarrow B) \mathrm{~iff~}  (\exists y, z \in S) [ R_{2}xyz \mathrm{~and~} T \in v_{\mathcal{M}}(y, A) \mathrm{~and~} F \in v_{\mathcal{M}}(z, B)]$.\footnote{Technically, this is our major departure from Routley's completion of the American plan. In our notation, Routley adopts this evaluation clause:
\begin{center}
$F \in v_{\mathcal{M}}(x, A \rightarrow B) \mathrm{~iff~}  (\exists y, z \in S) [ R_{2}xyz \mathrm{~and~} F \in v_{\mathcal{M}}(y, B) \mathrm{~and~} F \notin v_{\mathcal{M}}(z, A)]$.
\end{center}
Routley has in place some additional constraints that have no parallel here.} 
\end{enumerate}
\end{definition}


In \S\ref{motive} we said that $\le$ is an information ordering. We justify that remark with the Persistence Lemma (or Hereditary Condition).

\begin{lemma}[Persistence Lemma]\label{Persist}
In any model $\mathcal{M}$ and for any formula $A$,
\begin{quote}if $x \leq y$ and $T \in v_{\mathcal{M}}(x, A)$ then $T \in v_{\mathcal{M}}(y, A)$;\\
if $x \leq y$ and $F \in v_{\mathcal{M}}(x, A)$ then $F \in v_{\mathcal{M}}(y, A)$.
\end{quote}
\end{lemma}

\begin{proof}
\textit{Base case} This holds for atomic formulas by the definition of \emph{model}.\\
\textit{Induction hypothesis} Suppose that the lemma holds for all formulas of length $\leq k$.\\
\textit{Inductive step} Let $A$ be of length $k + 1$. There are four cases to consider in each of which $B$ and, where present, $C$ is of length $\leq k$. (We omit the subscript of the valuation function $v_{\mathcal{M}}$).
\begin{enumerate}[label=(\roman*)]
\item $A = \neg{}B$: $T \in v(x, \neg{}B)$ only if $F \in v(x, B)$ only if, by the induction hypothesis, $F \in v(y, B)$ only if $T \in v(y, \neg{}B)$;\\
$F \in v(x, \neg{}B)$ only if $T \in v(x, B)$ only if, by the induction hypothesis, $T \in v(y, B)$ only if $F \in v(y, \neg{}B)$;

\item $A = B \land C$: $T \in v(x, B \land C)$ only if $T \in v(x, B)$ and $T \in v(x, C)$ only if, by the induction hypothesis, $T \in v(y, B)$ and $T \in v(y, C)$ only if $T \in v(y, B \land C)$;\\
$F \in v(x, B \land C)$ only if $F \in v(x, B)$ or $F \in v(x, C)$ only if, by the induction hypothesis, $F \in v(y, B)$ or $F \in v(y, C)$ only if $v \in (y, B \land C)$;

\item $A = B \lor C$: $T \in v(x, B \lor C)$ only if $T \in v(x, B)$ or $T \in v(x, C)$ only if, by the induction hypothesis, $T \in v(y, B)$ or $T \in v(y,  C)$ only if $T \in v(y, B \lor C)$;\\
$F \in v(x, B \lor C)$ only if $F \in v(x, B)$ and $F \in v(x, C)$ only if, by the induction hypothesis, $F \in v(y, B)$ and $F \in v(y, C)$ only if $F \in v(y, B \lor C)$;

\item $A = B \to C$: $T \in v(x, B \to C)$ only if $\forall z, w$ [if $R_{1}xzw$ and $T \in v(z, B)$ then $T \in v(w, C)]$; if $R_{1}yzw$ then, by clause \textit{iv}) in Definition \ref{defFrame}, $R_{1}xzw$; consequently, if $T \in v(z, B)$ then $T \in v(w, C)]$; and so, $z$ and $w$ being arbitrary, we have that $\forall z, w$ [if $R_{1}yzw$ and $T \in v(z, B)$ then $T \in v(w, C)]$, \textit{i.e.}, $T \in v(y, B \to C)$;\\
$F \in v(x, B \to C)$ only if $\exists z, w [ R_{2}xzw$ and $T \in v(z, B)$ and $F \in v(w, C)]$; by clause \textit{v}) of Definition \ref{defFrame}, $\exists z, w [ R_{2}yzw$ and $T \in v(z, B)$ and $F \in v(w, C)]$; that is, $F \in v(y, B \to C)$.
\end{enumerate}

\end{proof}

Our definitions of logical consequence and logical truth are common in the literature on relevant logic. Here we make good on the thought that the well-behaved labs, not necessarily all labs, acknowledge the truth of laws of logic.

\begin{definition}[Logical Consequence \textit{and} Logical Truth]\label{defConsequence}
Given a model $\mathcal{M}$ determined by the frame $\mathscr{F} = \langle S, L, \leq, R_1, R_2 \rangle$ and valuation $v$ and a \emph{non-empty} set of formulas $X$, we write $X \vDash_{\mathcal{M}} A$ iff, at every $x \in S$,  $T \in v_{\mathcal{M}}(x, A)$ when, for all $B \in X$, $T \in v_{\mathcal{M}}(x, B)$. We write $\vDash_{\mathcal{M}} A$ iff, at every $x \in L$,  $T \in v_{\mathcal{M}}(x, A)$.

We then define \emph{logical consequence} and \emph{logical truth} as follows:
\begin{center}
$X \vDash A$ iff, for all models $\mathcal{M}$, $X \vDash_{\mathcal{M}} A$;\\
$\vDash A$ iff, for all models $\mathcal{M}$, $\vDash_{\mathcal{M}} A$.
\end{center}
\end{definition}

\noindent As this definition makes clear, in a model $\mathcal{M}$ the nodes in the distinguished subset $L$ of $S$ are states in which all logical truths hold; outside $L$ some logical truths may fail to be true. (Letting $t$ be the collection of all logical truths, we have, obviously and uninformatively, that $\vDash A$ iff $t \vDash A$.\footnote{Alternatively, following and adapting \cite{Read1988}, we could introduce a sentential constant $t$ with the constraint on valuations that $v(x,t) = \lbrace T \rbrace$ if $x \in L$, $v(x,t) = \emptyset$ if $x \notin L$. By appeal to the $\leq$-closure of $L$, we would find that $t$ satisfies the first clause of the Persistence Lemma; it satisfies the second clause trivially. We would then have Read's Proposition 5.3 (p.\ 85), \textit{i.e.}, $t \vDash A$ iff $~\vDash A$, and could say with Read, ``$t$ is the logic.''})


As is standard, the definitions of logical consequence and logical truth permit proof of a weak deduction theorem.

\begin{thm}[(Weak) Deduction Theorem]\label{WDT} For all formulas $A$ and $B$ in $\mathcal{L}$,
\begin{center}$A \vDash B$ iff $\vDash A \to B$.\label{DedEq}\end{center}
\end{thm}
\begin{proof}
It is enough to show that for an arbitrary model $\mathcal{M}$,
\begin{center}$A \vDash_\mathcal{M} B$ iff $\vDash_\mathcal{M} A \to B$.\end{center} 
Suppose first that $\vDash_\mathcal{M} A \to B$. Let $x \in S$ and $T \in v_\mathcal{M}(x, A)$.  Since, for all $x \in S$, $(\exists u \in L)R_{1}uxx$, we have that $T \in v_\mathcal{M}(u, A \to B)$, since $\vDash_\mathcal{M} A \to B$, and therefore that $T \in v_\mathcal{M}(x, B)$. Thus $A \vDash_\mathcal{M} B$.

Now suppose that $\nvDash_\mathcal{M} A \to B$ so that, for some $u \in L$, there exist $x$ and $y$ in $S$ such that $R_{1}uxy$, $T \in v_\mathcal{M}(x, A)$ and $T \notin v_{\mathcal{M}}(y, B)$. Now, since $u \in L$, $x \leq y$ and so, by the Persistence Lemma, $T \in v_\mathcal{M}(y, A)$. Hence $A \nvDash_\mathcal{M} B$.
\end{proof}

The definitions of Section \ref{LLR} which build in the constraints on the relations $R_1$ and $R_2$ motivated in section \ref{motive} yield what we are calling the Logic of Scientific Discovery (LScD). 


In Definition \ref{defConsequence} we have, in effect, taken $\lbrace T \rbrace$ and $\lbrace T, F \rbrace$ to be the designated values in the Hasse diagram (as foreshadowed in the shading above). The $\to$-free fragment of our logic is the familiar Dunn--Belnap four-valued logic, arguably what remains of the classical logic of negation, conjunction and disjunction when one gives up on the principles that truth and falsity are exhaustive (bivalence) and mutually exclusive (non-contradiction).\footnote{It's worth noting that, as a moment's reflection on the symmetry properties of the Hasse diagram reveals, nothing would change with regard to which inference patterns are sound if, instead of $\lbrace T \rbrace$ and $\lbrace T, F \rbrace$, we were to take $\lbrace T \rbrace$ and $\emptyset$ to be the designated values---{}avoidance of falsity rather than pursuit of truth then being the goal. Classical logic does not distinguish between these goals, of course, and, in a sense that may appeal to inferentialists, neither does what remains of the classical logic of negation, conjunction and disjunction when one gives up on the principles that truth and falsity are exhaustive and mutually exclusive.

This observation provides the basis for an easy proof that in Dunn--Belnap logic, the $\to$-free fragment of our logic, $B \vDash \neg{}A$ when $A \vDash \neg{}B$.}

While some may hold with Wittgenstein (\cite[\S{}290]{Wittgenstein1980}) that literally nothing follows from a contradiction and so there may be exceptions to $A \vDash A$, relevantists do not demur. As follows from Definition \ref{defConsequence} and Theorem \ref{DedEq}, $A \to A$ is a theorem for \emph{every} formula $A$ of $\mathcal{L}$. On the other hand, the logic of first-degree entailment gives us no $\to$-free theorems. In this regard, our logic really is a logic of regularities.


\subsection{Soundness}

\textbf{LScD} logics have axiom schemata A1 - A11 and possibly other of the following, depending on which constraints other than \textit{i}) -- \textit{v}) are adopted from Definition \ref{defFrame}:
\begin{enumerate}[label=A\arabic*]
\item $A \to A$;
\item $A \to (A \lor B)$, $B \to (A \lor B)$;
\item $(A \land B) \to A$, $(A \land B) \to B$;
\item $((A \to C) \land (B \to C)) \to ((A \lor B) \to C))$;
\item $((A \to B) \land (A \to C)) \to (A \to (B \land C))$;
\item $(A \land (B \lor C)) \to ((A \land B) \lor C)$;
\item $\neg(A \lor B) \to (\neg{}A \land \neg{}B)$, $(\neg{}A \land \neg{}B) \to \neg(A \lor B)$;
\item $\neg(A \land B) \to (\neg{}A \lor \neg{}B)$, $(\neg{}A \lor \neg{}B) \to \neg(A \land B)$;
\item $A \to \neg\neg{}A$, $\neg\neg{}A \to A$;
\item $\neg((A \lor B) \to C) \to (\neg(A \to C) \lor \neg(B \to C))$;
\item $\neg(A \to (B \land C)) \to (\neg(A \to B) \lor \neg(A \to C))$;
\item $(A \land (A \to B)) \to B$;
\item $(A \land \neg{}B) \to \neg(A \to B)$;
\item $((A \to B) \land (B \to C)) \to (A \to C)$;
\item $((A \rightarrow B) \land \neg(A \to C)) \rightarrow \neg(B \to C)$;
\item $((\neg{}A \rightarrow \neg{}B) \land \neg(C \to A)) \rightarrow \neg(C \to B)$
\end{enumerate}
\noindent and these rules of proof:
\begin{enumerate}[label=R\arabic*]
\item $A, A \to B / B$ (modus ponens);
\item $A, B / A \land B$ (adjunction);
\item $A \to B  / (C \to A) \to (C \to B)$ (prefixing);
\item $A \to B  / (B \to C) \to (A \to C)$ (suffixing);
\item $A \rightarrow B / \neg(A \to C) \rightarrow \neg(B \to C)$ (negated suffixing);
\item $\neg{}A \rightarrow \neg{}B / \neg(C \to A) \rightarrow \neg(C \to B)$ (negated prefixing).
\end{enumerate}

Of the axioms, only A4, A5 and A10 -- A16 require any work, the rest following \textit{via} the Deduction Theorem from Dunn--Belnap logic.\\
\

\begin{proof}[Soundness of A4] Suppose that $T \in v(x, (A \to C) \land (B \to C))$. Then $T \in v(x, A \to C)$ and $T \in v(x, B \to C)$; thus for all $ y, z \in S$,if $R_{1}xyz$ and $T \in v(y, A)$ then $T \in v(z, C)$ and if $R_{1}xyz$ and $T \in v(y, B)$ then $T \in v(z, C)$. But then, if $R_{1}xyz$ and $T \in v(y, A)$ or $T \in v(y, B)$ then $T \in v(z, C)$, hence if $R_{1}xyz$ and $T \in v(y, A \vee B)$ then $T \in v(z, C)$, \textit{i.e.}, $T \in v(x, (A \lor B) \to C)$. Thus, for arbitrary $\mathcal{M}$, $(A \to C) \land (B \to C) \vDash_{\mathcal{M}} (A \lor B) \to C$ and hence $(A \to C) \land (B \to C) \vDash (A \lor B) \to C$. By the Deduction Theorem, $\vDash ((A \to C) \land (B \to C)) \to ((A \lor B) \to C)$.
\end{proof}

\begin{proof}[Soundness of A5] Suppose that $T \in v(x, (A \to B) \land (A \to C))$. Then $T \in v(x, A \to B)$ and $T \in v(x, A \to C)$; thus $(\forall y, z \in S)[$ if $R_{1}xyz$ and $T \in v(y, A)$ then $T \in v(z, B)]$ and $(\forall y, z \in S)[$ if $R_{1}xyz$ and $T \in v(y, A)$ then $T \in v(z, C)]$. But then, $(\forall y, z \in S)[$ if $R_{1}xyz$ and $T \in v(y, A)$ then $T \in v(y, B)$ and $T \in v(z, C)]$, hence $(\forall y, z \in S)[$ if $R_{1}xyz$ and $T \in v(y, A)$ then $T \in v(z, B \land C)]$, \textit{i.e.}, $T \in v(x, A \to (B \land C))$. Thus for arbitrary $\mathcal{M}$, $(A \to B) \land (A \to C) \vDash_{\mathcal{M}} A \to (B \land C)$ and hence $(A \to B) \land (A \to C) \vDash A \to (B \land C)$. By the Deduction Theorem, $\vDash ((A \to B) \land (A \to C)) \to (A \to (B \land C))$.
\end{proof}

\begin{proof}[Soundness of A10] Suppose that $T \in v(x, \neg((A \lor B) \to C))$. Then $(\exists y, z \in S)[ R_{2}xyz$ and $T \in v(y, A \lor B)$ and $F \in v(z, C)]$. But then, $(T \in v(y, A)$ or $T \in v(y, B))$ and $F \in v(z, C)$, hence $T \in v(y, A)$ and $F \in v(z, C)$ or $T \in v(y, B)$ and $F \in v(z, C)$, \textit{i.e.}, $T \in v(x, \neg(A \to C))$ or $T \in v(x, \neg(B \to C))$, so $T \in v(x, \neg(A \to C) \lor \neg(B \to C))$. Thus for arbitrary $\mathcal{M}$, $\neg((A \lor B) \to C) \vDash_{\mathcal{M}} \neg(A \to C) \lor \neg(B \to C)$ and hence $\neg((A \lor B) \to C) \vDash \neg(A \to C) \lor \neg(B \to C)$. By the Deduction Theorem, $\vDash \neg((A \lor B) \to C) \to (\neg(A \to C) \lor \neg(B \to C))$.
\end{proof}

\begin{proof}[Soundness of A11] Suppose that $T \in v(x, \neg(A \to (B \land C)))$. Then $(\exists y, z \in S)[ R_{2}xyz$ and $T \in v(y, A)$ and $F \in v(z, B \land C)]$. But then, $T \in v(y, A)$ and $F \in v(z, B)$ or $F \in v(z, C)$, hence $T \in v(y, A)$ and $F \in v(z, B)]$ or $T \in v(y, A)$ and $F \in v(z, C)]$, \textit{i.e.}, $T \in v(x, \neg(A \to B))$ or $T \in v(x, \neg(A \to C))$, so $T \in v(x, \neg(A \to B) \lor \neg(A \to C))$. Thus for arbitrary $\mathcal{M}$, $\neg(A \to (B \land C)) \vDash_{\mathcal{M}} \neg(A \to B) \lor \neg(A \to C)$ and hence $\neg(A \to (B \land C)) \vDash \neg(A \to B) \lor \neg(A \to C)$. By the Deduction Theorem, $\vDash \neg(A \to (B \lor C)) \to (\neg(A \to B) \lor \neg(A \to C))$.
\end{proof}

\begin{proof}[Soundness of A12] Here we appeal to constraint vi) on $R_1$. Suppose that $T \in v(x, A)$ and $T \in v(x, A \to B)$. Since $R_{1}xxx$, $v(x, B)$. Thus for arbitrary $\mathcal{M}$, $A \land (A \to B) \vDash_\mathcal{M} B$. Hence $A \land (A \to B) \vDash B$. By the Deduction Theorem, $\vDash (A \land (A \to B)) \to B$. 
\end{proof}

\begin{proof}[Soundness of A13] Here we appeal to constraint vii) on $R_2$. Suppose that $T \in v(x, A \land \neg{}B)$. Then $T \in v(x, A)$ and $F \in v(x, B)$. As $R_{2}xxx$ we have that $(\exists y, z)[R_{2}xyz$ and $T \in v_\mathcal{M}(y, A)$ and $F \in v(z, B)]$. Hence $F \in v(x, A \to B)$ and so $T \in v(x, \neg(A \to B))$. Thus for arbitrary $\mathcal{M}$, $A \land \neg{}B \vDash_\mathcal{M}  \neg(A \to B)$.  Hence $A \land \neg{}B \vDash  \neg(A \to B)$ and so, by the Deduction Theorem, $\vDash (A \land \neg{}B) \to \neg(A \to B)$.
\end{proof}

\begin{proof}[Soundness of A14] Here we appeal to constraint viii) on $R_{1}$.  Suppose that $T \in v(x, A \to B)$ and $T \in v(x, B \to C)$. Let $Rxyz$. By viii), there's a node $w$ such that $R_{1}xyw$ and $R_{1}xwz$. So now, if $T \in v(y,A)$ then $T \in v(w,B)$ and thus $T \in v(z,C)$. So $(\forall y, z)[$ if $R_{1}xyz$ and $T \in v_\mathcal{M}(y, A)$ then $T \in v(z, C)]$, \textit{i.e.}, $T \in v(x, A \to C)$. Hence for arbitrary $\mathcal{M}$, $(A \to B) \land (B \to C) \vDash_\mathcal{M} A \to C$ and so $(A \to B) \land (B \to C) \vDash A \to C$. By the Deduction Theorem, $\vDash ((A \to B) \land (B \to C)) \to (A \to C)$.
\end{proof}

\begin{proof}[Soundness of A15] Here we appeal to constraint ix) on $R_1$ and $R_{2}$.  Suppose that $T \in v(x, A \to B)$ and $T \in v(x, \neg(A \to C))$. Let $y$ and $z$ be such that $R_{2}xyz$, $T \in v(y, A)$ and $F \in v(z, C)$. By ix), there's a node $w$ such that $R_{1}xyw$ and $R_{2}xwz$. As $T \in v(y, A)$, $T \in v(w, B)$. So $w$ and $z$ are such that $R_{2}xwz$, $T \in v(w, B)$ and $F \in v(z, C)$. Thus $T \in v(x, \neg(B \to C))$. Hence for arbitrary $\mathcal{M}$, $(A \to B) \land \neg(A \to C) \vDash_{\mathcal{M}} \neg(B \to C)$ and so $\mathcal{M}$, $(A \to B) \land \neg(A \to C) \vDash \neg(B \to C)$. By the Deduction Theorem, $\vDash ((A \rightarrow B) \land \neg(A \to C)) \rightarrow \neg(B \to C)$.
\end{proof}

\begin{proof}[Soundness of A16] Here we appeal to constraint x) on $R_1$ and $R_{2}$.  Suppose that $T \in v(x, \neg{}A \to \neg{}B)$ and $T \in v(x, \neg(C \to A))$. Let $y$ and $z$ be such that $R_{2}xyz$, $T \in v(y, C)$ and $F \in v(z, A)$, \textit{i.e.}, $T \in v(z, \neg{}A)$. By x), there's a node $w$ such that $R_{2}xyw$ and $R_{1}xzw$. As $T \in v(z, \neg{}A)$, $T \in v(w, \neg{}B)$, \textit{i.e.}, $F \in v(w,B)$. So $y$ and $w$ are such that $R_{2}xyw$, $T \in v(y, C)$ and $F \in v(w, B)$. Thus $T \in v(x, \neg(C \to B))$. Hence for arbitrary $\mathcal{M}$, $(\neg{}A \to \neg{}B) \land \neg(C \to A) \vDash_{\mathcal{M}} \neg(C \to B)$ and so $(\neg{}A \to \neg{}B) \land \neg(C \to A) \vDash \neg(C \to B)$. By the Deduction Theorem, $\vDash ((\neg{}A \to \neg{}B) \land \neg(C \to A)) \to \neg(C \to B)$.
\end{proof}

For the rules we proceed as follows making heavy use of the Persistence Lemma (Lemma \ref{Persist}) and the Deduction Theorem (Theorem \ref{WDT}).

\begin{proof}[Soundness of R1] Suppose that, in some model ${\mathcal{M}}$, $\vDash_{\mathcal{M}} A$ and $\vDash_{\mathcal{M}} A \to B$. By (the proof of) the Deduction Theorem, $A \vDash_{\mathcal{M}} B$, \textit{i.e.}, for all $\forall x \in S [$ if $T \in v_{\mathcal{M}}(x, A)$ then $T \in v_{\mathcal{M}}(x, B)]$. Let $u \in L$. Then $T \in v_{\mathcal{M}}(u, A)$, hence $T \in v_{\mathcal{M}}(u, B)$. As this holds for all $u \in L$, $\vDash_{\mathcal{M}} B$. Thus $\vDash B$ when $\vDash A$ and $\vDash A \to B$.
\end{proof}

\begin{proof}[Soundness of R2] Suppose that $\vDash_{\mathcal{M}} A$ and $\vDash_{\mathcal{M}} B$. Let $u \in L$. Then $T \in v_{\mathcal{M}}(u, A)$ and $T \in v_{\mathcal{M}}(u, B)$. Hence $T \in v_{\mathcal{M}}(u, A \land B)$. As $u$ is an arbitrary member of $L$, $\vDash_{\mathcal{M}} A \land B$. Thus $\vDash A \land B$ when $\vDash A$ and $\vDash B$.
\end{proof}

\begin{proof}[Soundness of R3] Suppose that $\vDash_{\mathcal{M}} A \to B$. By (the proof of) the Deduction Theorem, $A \vDash_{\mathcal{M}} B$. Now, let $u \in L$ and $R_{1}uxy$, so that $x \leq y$. If $T \in v_{\mathcal{M}}(x, C \to A)$, by the Persistence Lemma,  $T \in v_{\mathcal{M}}(y, C \to A)$ so that for all $z, w \in S$, if $R_{1}yzw$ and $T \in v_{\mathcal{M}}(z, C)$ then $T \in v_{\mathcal{M}}(w, A)$. As  $A \vDash_{\mathcal{M}} B$, $T \in v_{\mathcal{M}}(w, B)$. And thus $T \in v_{\mathcal{M}}(y, C \to B)$. Putting the pieces together, for all $u$ in $L$, $T \in v_{\mathcal{M}}(u, (C \to A) \to (C \to B))$, \textit{i.e.}, $\vDash_{\mathcal{M}}  (C \to A) \to (C \to B)$. Thus $\vDash  (C \to A) \to (C \to B)$ when $\vDash A \to B$.
\end{proof}

\begin{proof}[Soundness of R4] Suppose that $ \vDash_{\mathcal{M}} A \to B$. By (the proof of) the Deduction Theorem, $A \vDash_{\mathcal{M}} B$. Now, let $u \in L$ and $R_{1}uxy$, so that $x \leq y$. If $T \in v_{\mathcal{M}}(x, B \to C)$, by the Persistence Lemma,  $T \in v_{\mathcal{M}}(y, B \to C)$. Now, for all $z, w \in S$, if $R_{1}yzw$ and $T \in v_{\mathcal{M}}(z, A)$ then, since $A \vDash_{\mathcal{M}} B$, $T \in v_{\mathcal{M}}(z, B)$; since $v_{\mathcal{M}}(y, B \to C)$, $T \in v_{\mathcal{M}}(w, C)$. And thus $T \in v_{\mathcal{M}}(y, A \to C)$. Putting the pieces together, for all $u$ in $L$, $T \in v_{\mathcal{M}}(u, (B \to C) \to (A \to C))$, \textit{i.e.}, $\vDash_{\mathcal{M}} (B \to C) \to (A \to C)$. Thus $\vDash (B \to C) \to (A \to C)$ when $ \vDash A \to B$.
\end{proof}

\begin{proof}[Soundness of R5] Suppose that $\vDash_{\mathcal{M}} A \to B$. By (the proof of) the Deduction Theorem, $A \vDash_{\mathcal{M}} B$. Now, let $u \in L$ and $R_{1}uxy$, so that $x \leq y$. If $T \in v_{\mathcal{M}}(x, \neg(A \to C))$ then, by the Persistence Lemma,  $T \in v_{\mathcal{M}}(y, \neg(A \to C))$ so that there exists $z$ and $w$ such that $R_{2}yzw$, $T \in v_{\mathcal{M}}(z, A)$ and $F \in v_{\mathcal{M}}(w, C)$. Now, as  $A \vDash_{\mathcal{M}} B$, there exists $z$ and $w$ such that $R_{2}yzw$, $T \in v_{\mathcal{M}}(z, B)$ and $F \in v_{\mathcal{M}}(w, C)$. Hence $T \in v_{\mathcal{M}}(y, \neg(B \to C))$. Putting the pieces together, for all $u$ in $L$, $T \in v_{\mathcal{M}}(u, \neg(A \to C) \to \neg(B \to C))$, \textit{i.e.}, $\vDash_{\mathcal{M}} \neg(A \to C) \to \neg(B \to C)$. Thus $\vDash \neg(A \to C) \to \neg(B \to C)$ when $\vDash A \to B$.
\end{proof}

\begin{proof}[Soundness of R6] Suppose that $\vDash_{\mathcal{M}} \neg{}A \to \neg{}B$. By (the proof of) the Deduction Theorem, $\neg{}A \vDash_{\mathcal{M}} \neg{}B$. Now, let $u \in L$ and $R_{1}uxy$, so that $x \leq y$. If $T \in v_{\mathcal{M}}(x, \neg(C \to A))$ then, by the Persistence Lemma,  $T \in v_{\mathcal{M}}(y, \neg(C \to A))$ so that there exists $z$ and $w$ such that $R_{2}yzw$, $T \in v_{\mathcal{M}}(z, C)$ and $F \in v_{\mathcal{M}}(w, A)$, \textit{i.e.}, $T \in v_{\mathcal{M}}(w, \neg{}A)$. Now, as  $\neg{}A \vDash_{\mathcal{M}} \neg{}B$, there exists $z$ and $w$ such that $R_{2}yzw$, $T \in v_{\mathcal{M}}(z, C)$ and $T \in v_{\mathcal{M}}(w, \neg{}B)$, \textit{i.e.}, $F \in v_{\mathcal{M}}(w, B)$. Hence $T \in v_{\mathcal{M}}(y, \neg(C \to B))$. Putting the pieces together, for all $u$ in $L$, $T \in v_{\mathcal{M}}(u, \neg(C \to A) \rightarrow \neg(C \to B))$, \textit{i.e.}, $\vDash_{\mathcal{M}} \neg(C \to A) \rightarrow \neg(C \to B)$. Thus $\vDash \neg(C \to A) \rightarrow \neg(C \to B)$ when $\vDash \neg{}A \to \neg{}B$.
\end{proof}

\subsection{Completeness}
Our basic system consists of the axioms A1 -- A11 and rules R1 -- R6. This system can be extended by some (or all) of A12 to A16.

As usual a \emph{proof} is a finite sequence of formulas of $\mathcal{L}$ such that every formula in the sequence is either an instance of one of the axiom schemata in play or is obtained from previous members of the sequence by application of rules R1 -- R6. We write $\vdash A$ if there is a proof whose last member is $A$.

\begin{lemma}\label{assoccom} In the basic system, hence in all considered here, $\vdash A \to (A \land A)$; $\vdash (A \land A) \to A$; $\vdash A \to (A \lor A)$; $\vdash (A \lor A) \to A$; $\vdash (A \land B) \to (B \land A)$; $\vdash (A \land (B \land C)) \to ((A \land B) \land C)$; $\vdash ((A \land B) \land C) \to (A \land (B \land C))$; $\vdash (A \lor B) \to (B \lor A)$; $\vdash (A \lor (B \lor C)) \to ((A \lor B) \lor C)$; $\vdash ((A \lor B) \lor C) \to (A \lor (B \lor C))$.
\end{lemma}
\begin{proof} Axioms A1, A3 and A5 and rules R1, R2, and R3 ensure the idempotency, associativity and commutativity of conjunction; axioms A1, A2 and A4 and rules R1, R2, and R4 ensure the idempotency, associativity and commutativity of disjunction.
\end{proof}

\begin{definition}[Theories]~\ 
\textcolor{blue}{Let $X, Y$ be a non-empty sets} of formulas, then 
\begin{enumerate}[label=\alph*)]
\item we write $X \vdash B$ iff there are formulas $A_{1}, A_{2}, \ldots, A_{n} \in X, n > 0$, such that $\vdash(A_{1} \wedge A_{2} \wedge \ldots \wedge A_{n}) \to B$;

\item we write $X \vdash Y$ iff $(\exists B_{1}, B_2, \ldots, B_n \in Y) ~X \vdash B_1 \lor B_2 \lor \ldots \lor B_n$;

\item $X$ is a \emph{theory} iff, for all $A \in {\mathcal{L}}$, $A \in X$ when $X \vdash A$;

\item a theory $X$ is \emph{prime} iff, for all $A, B \in {\mathcal{L}}$,  $A \vee B \in X$ iff $A \in X$ or $B \in X$;

\item a theory $X$ is \emph{proper} iff $X \neq \mathcal{L}$;

\item a theory $X$ is \emph{logical} iff, for all of A1 -- A11 and whichever of A12 --A16 we may choose to add, every instance in $\mathcal{L}$ belongs to $X$.  (This makes `logical' a relative notion.)
\end{enumerate}
\end{definition}

\begin{lemma}[Properties of theories]\label{propth} Let $X$ be a theory. Then
\begin{enumerate}[label=\roman*)]
\item $X \vdash A$ iff $A \in X$;
\item \label{nondisj} if $Y$ is a set of formulas, then $X \vdash Y$ if $X \cap Y \neq \emptyset$;

\item \label{DT} if $A \in X$ and $\vdash A \to B$ then $B \in X$;

\item \label{conj} $X$ is closed under conjunction;

\item \label{clostheory} the deductive closure of a set of sentences is a theory, \textit{i.e.}, where $Y \subseteq \mathcal{L}$, $\lbrace A \in \mathcal{L} : Y \vdash A \rbrace$ is a theory.

\end{enumerate}

\end{lemma}
\begin{proof}
\begin{enumerate}[label=\roman*)]
\item `If' from c) in the definition above. `Only if' from axiom schema A1.
\item and iii) follow trivially.
\end{enumerate}
\begin{enumerate}[label=\roman*)]
\addtocounter{enumi}{3}
\item Let $A, B \in X$. By (A1), $\vdash (A \land B) \to (A \land B)$; as $A, B\in X$, $X \vdash A\wedge B$ by definition. As $X$ a theory, $A \land B \in X$.

\item Let $Z = \lbrace A \in \mathcal{L} : Y \vdash A \rbrace$. If $Z \vdash A$ then $(\exists B_1, \ldots, B_n \in Z) ~\vdash (B_1 \land \ldots \land B_n) \to A$. As $Y \vdash B_i$ for each $B_i$, there are $ C_1^i, \ldots, C_{i_{m}}^i \in Y$ such that $\vdash C_1^i  \land \ldots \land C_{i_{m}}^i \to B_i$. By A3 and R4 (suffixing), and appealing heavily to Lemma \ref{assoccom}, for each $i$, $1 \leq i \leq n$, $~\vdash (C_1^1 \land \ldots \land C_{i_{1}}^1 \land C_1^2 \land \ldots \land C_{i_{2}}^2 \land \ldots \land C_1^n \land \ldots \land C_{i_{n}}^n ) \to B_i$; by R2, A5, and R1, $~\vdash (C_1^1 \land \ldots \land C_{i_{1}}^1 \land \ldots \land C_1^n \land \ldots \land C_{i_{n}}^n ) \to (B_1 \land \ldots \land B_n)$ and by R4 (suffixing), $~\vdash (C_1^1 \land \ldots \land C_{i_{1}}^1 \land \ldots \land C_1^n \land \ldots \land C_{i_{n}}^n ) \to A$. Hence $Y \vdash A$ and $A \in Z$.
\end{enumerate}
\end{proof}


\begin{lemma} If $\vdash A \to (B \lor C)$ and $\vdash (D \land C) \to E$ then $\vdash (A \land D) \to (B \lor E)$.
\end{lemma}
\begin{proof} By A3, R1 and R4, $\vdash (A \land D) \to (B \lor C)$; by A3, A5, R1 and R2, $\vdash (A \land D) \to ((B \lor C) \land D)$. By A2, A3, A4, A5, R1 and R2, $\vdash ((B \lor C) \land D) \to (D \land (C \lor B))$. By R1 and R3, $\vdash (A \land D) \to (D \land (C \lor B))$. By A6, R1 and R3, $\vdash (A \land D) \to ((D \land C) \lor B)$. By A2 and R3, $\vdash (D \land C) \to (B \lor E)$. By A2, A4, R1 and R2, $\vdash ((D \land C) \lor B) \to (B \lor E)$. By R1 and R3, $\vdash (A \land D) \to (B \lor E)$.
\end{proof}

\begin{cor}[Cut] If $X \cup \lbrace C \rbrace \vdash Y$ and $X^{\prime} \vdash Y^{\prime} \cup \lbrace C \rbrace$ then $X \cup X^{\prime} \vdash Y \cup Y^{\prime}$.
\end{cor}
\begin{proof} Suppose that $X \cup \lbrace C \rbrace \vdash Y$ and $X^{\prime} \vdash Y^{\prime} \cup \lbrace C \rbrace$. Then there are  $A_{1}, \ldots, A_m \in X$  and $B_{1}, \ldots, B_n\in Y$ such that $\vdash (A_1 \land \ldots \land A_m \land C) \to (B_1 \lor \ldots \lor B_n)$. Similarly there are $ A_{1}^{\prime}, , \ldots, A_k^{\prime} \in X^{\prime}$ and $ B_{1}^{\prime}, \ldots, B_l^{\prime} \in Y$ such that $ ~\vdash (A_1^{\prime} \land \ldots \land A_k^{\prime}) \to (B_1^{\prime} \lor \ldots \lor B_l^{\prime} \lor C)$.  By the previous lemma, $\vdash (  A_1 \land \ldots \land A_m \land A_1^{\prime}  \land \ldots \land A_k^{\prime}) \to (B_1 \lor \ldots \lor B_n \lor B_1^{\prime} \lor \ldots \lor B_l^{\prime}  )$. Hence $X \cup X^{\prime} \vdash Y \cup Y^{\prime}$.
\end{proof}

\begin{thm}[Lindenbaum's Lemma] If we have two sets of formulas  $X, Y$ such that $X \nvdash Y$ then we can extend them to $X^{\prime}, Y^{\prime}$ such that $X \subseteq X^{\prime}, Y \subseteq Y^{\prime},X^{\prime} \nvdash Y^{\prime} $ and their union is the whole language $ X^{\prime} \cup Y^{\prime} = \mathcal{L}, $. Moreover $X^{\prime}$ is a prime theory and $Y^{\prime}$ is closed under disjunction.
\end{thm}
\begin{proof} Let $A_{1}, \ldots, A_n \ldots$ be an enumeration of $\mathcal{L}$. Let
\begin{align*}
X_0 &= X, & Y_0 &= Y\\
X_{n + 1} &= X_{n} \cup \lbrace A_n \rbrace, & Y_{n + 1} &= Y_n ~~~~~~if~ X_n \cup \lbrace A_n \rbrace \nvdash Y_n\\\
X_{n + 1} &= X_n, & Y_{n + 1} &= Y_n \cup \lbrace A_n \rbrace ~~~~~~if~ X_n \cup \lbrace A_n \rbrace \vdash Y_n\\
X^{\prime} &= \bigcup\limits_{n \in \mathbb{N}} X_n & Y^{\prime} &= \bigcup\limits_{n \in \mathbb{N}} Y_n
\end{align*}
Obviously, $X \subseteq X^{\prime}, Y \subseteq Y^{\prime}$ and $X^{\prime} \cup Y^{\prime} = \mathcal{L}$. We show by induction that $X_n \nvdash Y_n$ for all $n$. By definition $X_0 \nvdash Y_0$.  Suppose that $X_n \nvdash Y_n$ and $X_{n + 1} \vdash Y_{n + 1}$. If $X_{n} \cup \lbrace A_n \rbrace \nvdash Y_n$ then by the construction $X_{n + 1} = X_{n} \cup \lbrace A_n \rbrace$ and $Y_{n + 1} = Y_n$. Then $X_{n + 1} \vdash Y_{n + 1}$ implies   $X_{n} \cup \lbrace A_n \rbrace \vdash Y_n$ contrary to what we assumed. So it must be the case that  $X_{n} \cup \lbrace A_n \rbrace \vdash Y_n$ and by construction $X_{n + 1} = X_n$, $Y_{n + 1} = Y_n \cup \lbrace A_n \rbrace$. Hence $X_{n} \cup \lbrace A_n \rbrace \vdash Y_n$ and $X_n \vdash Y_n \cup \lbrace A_n \rbrace$. By Cut, $X_n \vdash Y_n$, a contradiction. Hence for all $n \in \mathbb{N}$, $X_n \nvdash Y_n$. But now, if $X^{\prime} \vdash Y^{\prime}$ then, by the finiteness of proof, for some $n \in \mathbb{N}$, $X_{n + 1} \vdash Y_{n + 1}$, which we have just shown is not possible. Hence $X^{\prime} \nvdash Y^{\prime}$.

Suppose that $X^{\prime} \vdash A$ and $A \notin X^{\prime}$. As $X^{\prime} \cup Y^{\prime} = \mathcal{L}$, $A \in Y^{\prime}$.
But then $X^{\prime} \vdash Y^{\prime}$ which we have just shown not to be the case. Thus $A \in X^{\prime}$ and $X^{\prime}$ is a theory.

Suppose that $A \lor B \in X^{\prime}$, $A \notin X^{\prime}$ and $B \notin X^{\prime}$. As $X^{\prime} \cup Y^{\prime} = \mathcal{L}$, $A \in Y^{\prime}$ and $B \in Y^{\prime}$. 
So $X^{\prime} \vdash Y^{\prime}$ which we have just shown not to be the case. Thus $A \in X^{\prime}$ or $B \in X^{\prime}$. $X^{\prime}$ is prime. (Notice that since $X^{\prime} \cup Y^{\prime} = \mathcal{L}$ and, by Lemma \ref{propth} \ref{nondisj}, $X^{\prime} \cap Y^{\prime} = \emptyset$, $X^{\prime}$ is prime iff $Y^{\prime}$ is closed under disjunctions.)
\end{proof}

\noindent (Thanks to the appeal to Cut, in deriving Lindenbaum's Lemma we have used axiom schemata A1, A2, A3, A4, A5 and A6 and rules R1, R2, R3, and R4.)

\begin{cor}\label{LindCor} Let $X$ be a theory and $Y$ a set of formulas disjoint from $X$ and closed under disjunction. Then there is a prime theory $X^{\prime}$ such that $X \ssq X^{\prime}$ and $X^{\prime}$ is disjoint from $Y$.
\end{cor}
\begin{proof}
If $X \vdash Y$ then, for some $B_1, B_2, \ldots, B_n \in Y$, $X \vdash B_1 \vee B_2 \vee \ldots \vee B_n$. But then $B_1 \vee B_2 \vee \ldots \vee B_n \in X$, as $X$ is a theory and $B_1 \vee B_2 \vee \ldots \vee B_n \in Y$ as $Y$ is closed under disjunction so, contrary to hypothesis, $X \cap Y \neq \emptyset$. So $X \nvdash Y$ and, by Lindenbaum's Lemma, $\exists X^{\prime}, Y^{\prime}$ such that $X \subseteq X^{\prime}, Y \subseteq Y^{\prime}, X^{\prime} \cup Y^{\prime} = \mathcal{L}, X^{\prime} \nvdash Y^{\prime}$ and $X^{\prime}$ is a prime theory. By Lemma \ref{propth} \ref{nondisj}, $X^{\prime}$ and $Y^{\prime}$ are disjoint, hence $X^{\prime}$ is disjoint from $Y$.
\end{proof}

\begin{definition}[Operations on sets of sentences] ~\ Let $X$ and $Y$ be non-empty sets of formulas. Then
\begin{enumerate}[label=\alph*)]
\item For subsets $X$ and $Y$ of $\mathcal{L}$, $X \ogreaterthan Y = \lbrace C \in \mathcal{L} : (\exists A \in \mathcal{L})[ A \to C \in X \mathit{~and~} A \in Y ] \rbrace$.
\item For subsets $X$ and $Y$ of $\mathcal{L}$, $X \ovee Y = \lbrace \neg(A \to  C): A \in X, \neg{}C \in Y \rbrace$.
\end{enumerate}
\end{definition}

\begin{lemma}[A fact about $\ogeq$]\label{ogeqth} When $X$ and $Y$ are theories, $X \ogreaterthan Y$ is a theory.
\end{lemma}
\begin{proof}
We first show that when $X$ and $Y$ are theories, $X \ogreaterthan Y$ is closed under conjunction.

If $C, D \in X \ogreaterthan Y$ then $(\exists E, F \in Y) [ E \to C \mathrm{~and~}F \to D \in X] $. By A3 and R4 (suffixing), $X \vdash (E \land F) \to C$ and $X \vdash (E \land F) \to D$. As $X$ is a theory, $(E \land F) \to C \in X$ and $(E \land F) \to D \in X$; by A5, $X \vdash (E \land F) \to (C \land D)$; as $X$ is a theory, $(E \land F) \to (C \land D) \in X$. By Lemma \ref{conj}, $E \land F \in Y$. Hence $C \land D \in X \ogreaterthan Y$.

We now show that when $X$ and $Y$ are theories, $X \ogreaterthan Y$ is deductively closed.

If $X \ogreaterthan Y \vdash C$ then, making use of what we have just shown, $(\exists E \in X \ogreaterthan Y) ~\vdash E \to C$. As $E \in X \ogreaterthan Y$,  $(\exists B \in Y) B \to E \in X$. By R3 (prefixing), $\vdash (B \to E) \to (B \to C)$, hence $X \vdash B \to C$ and so $B \to C \in X$. But then $C \in X \ogreaterthan Y$.

By Lemma \ref{propth} \ref{clostheory}, $X \ogreaterthan Y$ is a theory.
\end{proof}

\subsection{Completeness: the canonical model}

Now we are ready to build our canonical model. As usual, the domain $S$ of the \emph{canonical model} comprises all proper, prime theories in $\mathcal{L}$, theoryhood being relative to the logic in play.

\begin{definition}
Let $\mathbb{L}$ be a logic in a language $\mathcal{L} $ axiomatized by the axiom schemata A1-A11 and the rules R1-R6, possibly with some of the axiom schemata A12-A16. We denote by $\vdash_\mathbb{L}$ the corresponding provability relation. We define our canonical frame  $\mathcal{F}_\mathbb{L} = \langle S, L, \leq, R_1,R_2\rangle$ over the domain $S$ of all proper prime theories with the canonical relations defined as follows:

\begin{enumerate}[label=\roman*)]
\item $L = \lbrace X \in S: \lbrace A \in \mathcal{L} :~\vdash_\mathbb{L} A \rbrace \subseteq X \rbrace$.
\item For $X, Y \in S$, $Y \leq Z$ iff $Y \subseteq Z$.
\item For theories $X, Y, Z$, $R_{1}XYZ$ iff $X \ogeq Y \subseteq Z$.
\item For theories $X, Y, Z$, $R_{2}XYZ$ iff $Y \ovee Z \subseteq X$.
\end{enumerate}
\end{definition}

We should check that the canonical frame is indeed a frame, \textit{i.e.}, it satisfies the conditions of Definition \ref{defFrame}.

\begin{lemma}[Canonical frame]\label{Canonicalframelemma}
A canonical frame $\mathcal{F}_\mathbb{L}$ satisfies the conditions  of Definition \ref{defFrame} relevant to the logic $\mathbb{L}$.
\end{lemma}
\begin{proof}
We must first check that, for all $X$, $Y$ in $S$, $X \ssq Y$ iff $(\exists U \in L)R_{1}UXY$. Suppose, first, that $X \ssq Y$. Let $t = \lbrace C \in \mathcal{L} :~\vdash_\mathbb{L} C \rbrace$. If $A \to B \in t$ and $A \in X$ then $\vdash_\mathbb{L} A \to B$; by Lemma \ref{propth} \ref{DT}, $B \in X$ hence $B \in Y$. Thus $t \ogeq X \ssq Y$, \textit{i.e.}, $R_{1}tXY$.

Now we show that when $X, Y$ and $Z$ are proper theories, $R_{1}XYZ$ and $Z$ is prime, there is a proper, prime theory $X^{\prime}$ such that $X \ssq X^{\prime}$ and $R_{1}X^{\prime}YZ$. To begin, let $W = \lbrace A \in \mathcal{L} :$ for some $B \in Y$, $C \notin Z$, $A \vdash_\mathbb{L} B \to C \rbrace$. As $Z$ is proper, $W$ is non-empty. Let $A, B \in W$, so, for $C, D \in Y$, $E, F \notin Z$, $A \vdash_\mathbb{L} C \to E$ and $B \vdash_\mathbb{L} D \to F$. By prefixing and suffixing and appeals to axioms A2 and A3, $\vdash_\mathbb{L} A \to ((C \land D) \to (E \lor F))$ and $\vdash_\mathbb{L} B \to ((C \land D) \to (E \lor F))$; by $R2$, $A4$ and $R1$, $A \vee B \vdash_\mathbb{L} (C \land D) \to (E \lor F)$. $C \land D \in Y$ as $Y$ is a theory; $E \lor F \notin Z$ as $Z$ is prime. Hence $A \vee B \in W$. --- $W$ is closed under disjunction.

Suppose that $A \in X \cap W$. For some $B \in Y$, $C \notin Z$, $A \vdash_\mathbb{L} B \to C$. As $A \vdash_\mathbb{L} B \to C$, $X \vdash_\mathbb{L} B \to C$. As $X$ is a theory, $B \to C \in X$. But then $C \in Z$ as $R_{1}XYZ$. --- Contradiction. Thus $X \cap W = \emptyset$. By Corollary \ref{LindCor}, there is a prime theory $X^{\prime}$ such that $X \ssq X^{\prime}$ and $X^{\prime}$ is disjoint from $W$. As $W \neq \emptyset$, $X^{\prime}$ is proper. Let $A \to B \in X^{\prime}$. If $A \in Y$ then $A \to B \notin W$, hence $B \in Z$. Thus $X^{\prime}\ogeq Y \ssq Z$ and $R_{1}X^{\prime}YZ$.

And so there's a proper, prime theory $U$ such that $t \ssq U$ and $R_{1}UXY$; as $t \ssq U$, $U \in L$.

To show the converse, \textit{i.e.}, that $X \ssq Y$ if $(\exists U \in L)R_{1}UXY$, we note that if $U \in L$ then, for all $A \in \mathcal{L}$, $A \to A \in U$, hence $A \in U \ogeq X$ and thus $A \in Y$ since $R_{1}UXY$ when $A \in X$.

\begin{enumerate}[label=\roman*)]
\item and ii) (reflexivity and transitivity of $\leq $) are trivial.
\end{enumerate}
\begin{enumerate}[label=\roman*)]
\addtocounter{enumi}{2}
\item That $L$ is upward closed subset of $\langle S, \leq \rangle$ is immediate from the definition.

\item If $W \subseteq X$ then $W \ogeq Y \subseteq X \ogeq Y$, hence $W \ogeq Y \subseteq Z$ when $X \ogeq Y \subseteq Z$.
\item If $X \subseteq W$ then $Y \ovee Z \subseteq W$ when $Y \ovee Z \subseteq X$.


\item If for some $A \in X$, $A \to C \in X$ then, by Lemma \ref{propth} \ref{conj}, $A \land (A \to C) \in X$. According to A10, $\vdash_\mathbb{L} (A \land (A \to C)) \to C$, hence $X \vdash_\mathbb{L} C$; as $X$ is a theory, $C \in X$. Thus $R_1 XXX$ by definition, \emph{given axiom A10}.

\item Suppose that $A \in X$ and $\neg{}C \in X$. By Lemma \ref{propth} \ref{conj}, $A \land \neg{}C \in X$. According to A11, $\vdash_\mathbb{L} (A \land \neg{}C) \to \neg(A \to C)$ hence $X \vdash_\mathbb{L} \neg(A \to C)$; as $X$ is a theory, $\neg(A \to C) \in X$. Thus $R_2 XXX$ by definition, \emph{given axiom A11}.

\item Suppose that $R_{1}XYZ$. Then $R_{1}XY(X\ogeq{}Y)$, as follows from the definition of $R_1$, and $X\ogeq{}Y$ is proper as $Z$ is. Now let $B \in X\ogeq{}Y, B \to C \in X$. Then, for some $A \in Y$, $A \to B \in X$.  By Lemma \ref{propth} \ref{conj}, $(A \to B) \land (B \to C) \in X$. According to A12, $\vdash_\mathbb{L} ((A \to B) \land (B \to C)) \to (A \to C)$ hence $X \vdash_\mathbb{L} A \to C$; as $X$ is a theory, $A \to C \in X$. As $A \in Y$, $C \in X\ogeq{}Y$. Thus $R_{1}X(X\ogeq{}Y)(X\ogeq{}Y)$ by definition. As $R_{1}XYZ$, $X\ogeq{}Y \ssq Z$ and so $R_{1}X(X\ogeq{}Y)Z)$.

We show next that, for any proper theories $X, Y$ and $Z$, if $R_{1}XYZ$ and $Z$ is prime then there is a proper, prime theory $Y^{\prime}$ such that $Y \ssq Y^{\prime}$ and $R_{1}XY^{\prime}Z$. To begin, let $W = \lbrace A \in \mathcal{L} : (\exists B \in \mathcal{L})[ A \to B \in X$ and $B \notin Z]\rbrace$.  As $R_{1}XYZ$, $Y$ and $W$ are disjoint. Let $A, B \in W$, so, for some $C,  D \notin Z$, $A \to C, B \to D \in X$. As $Z$ is prime, $C \vee D \notin Z$. Also, by R3 (prefixing), $A \to (C \vee D), B \to (C \vee D) \in X$; by A3, R4, R2, A4, and R1, $(A \vee B) \to (C \vee D) \in X$, thus $A \vee B \in W$. --- $W$ is closed under disjunction. If $W = \emptyset$, we may take $Y^{\prime}$ to be any proper, prime extension of $Y$; as $W = \emptyset$, $R_{1}XY^{\prime}Z$. If $W \neq \emptyset$, then, by Corollary \ref{LindCor}, there is a (proper) prime theory $Y^{\prime}$ such that $Y \ssq Y^{\prime}$ and $Y^{\prime}$ is disjoint from $W$. Let $A \to B \in X$: if $A \in Y^{\prime}$ then $A \notin W$, hence $B \in Z$; thus $X\ogeq Y^{\prime} \ssq Z$ and $R_{1}XY^{\prime}Z$.\label{A12}

We have shown, \emph{given axiom A12}, that when $R_{1}XYZ$ there is a proper, prime theory $W$ such that $R_{1}XWZ$ and $R_{1}XYW$.

\item Suppose that $\neg(B \to C) \in (X\ogeq{}Y) \ovee Z$.  Then $B \in X\ogeq{}Y$ and $\neg{}C \in Z$. As $B \in X\ogeq{}Y$, for some $A \in Y$, $A \to B \in X$ and $\neg(A \to C) \in Y \ovee Z$.  As $R_{2}XYZ$, $\neg(A \to C) \in X$. By Lemma \ref{propth} \ref{conj}, $(A \to B) \land \neg(A \to C) \in X$. According to A13, $\vdash ((A \to B) \land \neg(A \to C)) \to \neg(B \to C)$ hence $X \vdash \neg(B \to C)$; as $X$ is a theory, $\neg(B \to C) \in X$. Thus $(X\ogeq{}Y) \ovee Z \subseteq X$, \textit{i.e.}, $R_{2}X(X\ogeq{}Y)Z$.

We show next that, for any proper theories $X, Y$ and $Z$, if $R_{2}XYZ$ and $X$ is prime, there is a proper, prime theory $Y^{\prime}$ such that $Y \ssq Y^{\prime}$ and $R_{2}XY^{\prime}Z$. To begin, let $W = \lbrace A \in \mathcal{L} : (\exists B \in \mathcal{L})[\neg{B} \in Z \mathrm{~and~} \neg(A \to B) \notin X]\rbrace$.   As $R_{2}XYZ$, $Y$ and $W$ are disjoint. Let $A, C \in W$, so, for some $\neg{B} \in Z$, $\neg(A \to B) \notin X$ and, for some $\neg{D} \in Z$, $\neg(C \to D) \notin X$. $\neg{}B \land \neg{}D \in Z$, as $Z$ is a theory; by A7, $\neg(B \lor D) \in Z$. As $\vdash \neg(B \lor D) \to \neg{}B$, $\vdash \neg(A \to (B \lor D)) \to \neg(A \to B)$ by R6. Hence $\neg(A \to (B \lor D)) \notin X$. Likewise $\neg(C \to (B \lor D)) \notin X$. As $X$ is prime, $\neg(A \to (B \lor D)) \lor \neg(C \to (B \lor D)) \notin X$. By axiom A10, $\vdash \neg((A \lor C) \to (B \lor D)) \to (\neg(A \to (B \lor D)) \lor \neg(C \to (B \lor D)))$, hence $\neg((A \lor C) \to (B \lor D))  \notin X$. Thus $A \lor C \in W$. --- $W$ is closed under disjunction. If $W = \emptyset$, we may take $Y^{\prime}$ to be any proper, prime extension of $Y$. As $W = \emptyset$, $R_{2}XY^{\prime}Z$. If $W \neq \emptyset$, then, by Corollary \ref{LindCor}, there is a prime theory $Y^{\prime}$ such that $Y \ssq Y^{\prime}$ and $Y^{\prime} \cap W = \emptyset$.  If $A \in Y^{\prime}$ then $A \notin W$, hence, for all $\neg{}B \in Z$, $\neg(A \to B) \in X$. Thus $Y^{\prime} \ovee Z \ssq X$ and so $R_{2}XY^{\prime}Z$.\label{A13}

We have shown, \emph{given axiom A13}, that when $R_{2}XYZ$ there is a proper, prime theory $W$ such that $R_{1}XYW$ and $R_{2}XWZ$.

\item Suppose that $R_{2}XYZ$. Let $A \in Y$, $\neg{}B \in X \ogeq Z$. So, for some $C \in Z, C \to \neg{}B \in X$. By A9 and R3, $\neg\neg{}C \to \neg{}B \in X$ and, by A9, $\neg\neg{}C \in Z$ (as $Z$ is a theory). As $R_{2}XYZ$, $\neg(A \to \neg{}C) \in X$. By Lemma \ref{propth} \ref{conj}, $(\neg\neg{}C \to \neg{}B) \land \neg(A \to \neg{}C) \in X$. According to A14, $\vdash ((\neg\neg{}C \to \neg{}B) \land \neg(A \to \neg{}C)) \to \neg(A \to B)$, hence $\neg(A \to B) \in X$. Thus $Y \ovee (X \ogeq Z) \ssq X$, \textit{i.e.}, $R_{2}XY(X \ogeq Z)$.

We show next that, for any proper theories $X, Y$ and $Z$, if $R_{2}XYZ$ and $X$ is prime, there is a proper, prime theory $Y^{\prime}$ such that $Y \ssq Y^{\prime}$ and $R_{2}XY^{\prime}Z$. To begin, let $W = \lbrace A \in \mathcal{L} : (\exists B \in \mathcal{L})[\neg{B} \in Z \mathrm{~and~} \neg(A \to B) \notin X]\rbrace$.   As $R_{2}XYZ$, $Y$ and $W$ are disjoint. Let $A, C \in W$, so, for some $\neg{B} \in Z$, $\neg(A \to B) \notin X$ and, for some $\neg{D} \in Z$, $\neg(C \to D) \notin X$. $\neg{}B \land \neg{}D \in Z$, as $Z$ is a theory; by A7, $\neg(B \lor D) \in Z$. As $\vdash \neg(B \lor D) \to \neg{}B$, $\vdash \neg(A \to (B \lor D)) \to \neg(A \to B)$ by R6. Hence $\neg(A \to (B \lor D)) \notin X$. Likewise $\neg(C \to (B \lor D)) \notin X$. As $X$ is prime, $\neg(A \to (B \lor D)) \lor \neg(C \to (B \lor D)) \notin X$. By axiom A10, $\vdash \neg((A \lor C) \to (B \lor D)) \to (\neg(A \to (B \lor D)) \lor \neg(C \to (B \lor D)))$, hence $\neg((A \lor C) \to (B \lor D))  \notin X$. Thus $A \lor C \in W$. --- $W$ is closed under disjunction. If $W = \emptyset$, we may take $Y^{\prime}$ to be any proper, prime extension of $Y$. As $W = \emptyset$, $R_{2}XY^{\prime}Z$. If $W \neq \emptyset$, then, by Corollary \ref{LindCor}, there is a prime theory $Y^{\prime}$ such that $Y \ssq Y^{\prime}$ and $Y^{\prime} \cap W = \emptyset$.  If $A \in Y^{\prime}$ then $A \notin W$, hence, for all $\neg{}B \in Z$, $\neg(A \to B) \in X$. Thus $Y^{\prime} \ovee Z \ssq X$ and so $R_{2}XY^{\prime}Z$.\label{A14} 

We have shown, \emph{given axiom A14}, that when $R_{2}XYZ$ there is a proper, prime theory $W$ such that $R_{2}XYW$ and $R_{1}XZW$.

\end{enumerate}
\end{proof}

Adding the canonical valuation defined for atomic $p$ in $\mathcal{L}$ as
\begin{center}
$T \in v(X, p)$ iff $p \in X$,\\
$F \in v(X, p)$ iff $\neg{}p \in X$
\end{center}
we obtain the canonical LScD model $\mathcal{M}$. So defined, $v$ automatically satisfies the pair of constraints
\begin{enumerate}[label=\roman*)]
\item if $X \leq Y$ and $T \in v(X, p)$ then $T \in v(Y, p)$;
\item if $X \leq Y$ and $F \in v(X, p)$ then $F \in v(Y, p)$.
\end{enumerate}
So $\langle \mathscr{F}, v \rangle$ is a model in the sense of Definition \ref{defModel}. 

\begin{lemma}[Valuation lemma] Given a canonical model $\mathcal{M_C} = \langle \mathscr{F_C}, v \rangle$, for all $X \in S$ and $A \in \mathcal{L}$:\\

\indent $T \in v_{\mathcal{M}}(X, A)$ iff $A \in X$ and\\ 
\indent $F \in v_{\mathcal{M}}(X, A)$ iff $\neg{}A \in X$.
\end{lemma}

In the inductive proof that $T \in v(X, A)$ iff $A \in X$, we skip the easy cases and attend only to conditionals and their negations. We have:
\begin{itemize}
\item If $A \to B \in X$ then, for any $Y, Z \in S$  such that $R_{1}XYZ$, if $A \in Y$ then $B \in Z$. By the induction hypothesis this means that $T \in v(Y, A)$ and $T \in v(Z,B)$, hence $T \in v(X, A \to B)$.

\item Suppose that $A \to B \notin X$. Let $Y = \lbrace B \in \mathcal{L} : A \vdash_{\mathbb{L}} B \rbrace$; let $Z = X \ogeq Y$.  By lemmata \ref{propth} \ref{clostheory} and \ref{ogeqth}, $Y$ and $Z$ are theories. Obviously, $A \in Y$. If $B \in Z = X \ogeq Y$, then $A \to B \in X$; as this is contrary to hypothesis, $B \notin Z$.  By Lindenbaum's Lemma, there is a proper, prime theory $Z^{\prime}$ such that $Z \subseteq  Z^{\prime}$ and $B \notin Z^{\prime}$ so $Z^{\prime}$ is proper. 

$X$ and $Z^{\prime}$ are proper, prime theories. As $R_{1}XYZ$, $R_{1}XYZ^{\prime}$ and so, as in the proof of Lemma \ref{Canonicalframelemma} \ref{A12}, there is a proper, prime theory $Y^{\prime}$ extending $Y$ such $R_{1}XY^{\prime}Z^{\prime}$, where $A \in Y^{\prime}$, $B \notin Z^{\prime}$. By the induction hypothesis, $T \in v(Y^{\prime}, A)$ and $T \notin v(Z^{\prime},B)$, hence $T \notin v(X, A \to B)$.

\medskip


\item Let $\neg(A \to B) \in X$. 
Let $Y = \lbrace C \in \mathcal{L} : A \vdash_{\mathbb{L}} C \rbrace$ and $Z = \lbrace C \in \mathcal{L} : \neg{}B \vdash_{\mathbb{L}} C \rbrace$. By Lemmata \ref{propth} \ref{clostheory}, $Y$ and $Z$ are theories. Let $C \in Y$ and $\neg{}D \in Z$. As $\vdash A \to C$, $\neg(C \to B) \in X$, by R5 (negated suffixing); as $\vdash \neg{}B \to \neg{}D$, $\neg(C \to D) \in X$, by R6 (negated prefixing). Thus $Y \ovee Z \ssq X$, \textit{i.e.}, $R_{2}XYZ$. As in the proof of Lemma \ref{Canonicalframelemma} \ref{A13} and \ref{A14}, there are \emph{proper, prime} theories $Y^\prime$ and $Z^\prime$ such that $Y \ssq Y^\prime$, $Z \ssq Z^\prime$ and $R_{2}XY^\prime{}Z^\prime$, where $A \in Y^{\prime}$, $\neg{}B \in Z^{\prime}$. By the IH we have $R_{2}XY^\prime{}Z^\prime$ and $T \in v(Y^\prime, A)$ and $F \in v(Z^\prime, B)$, hence $T \in v(X, \neg(A \to B))$.

\item if $\neg(A \to B) \notin X$ then, if $R_{2}XYZ$, $\neg(A \to B) \notin Y \ovee Z$ so $A \notin Y$ or $\neg{}B \notin Z$. By the IH we have $(\forall Y, Z \in S)[$if $R_{2}XYZ$ and $T \in v(Y, A)$ then $T \notin v(Z,\neg{}B)]$, \textit{i.e.}, it is not the case that $(\exists Y, Z \in S)[R_{2}XYZ$ and $T \in v(Y, A)$ and $F \in v(Z,B)]$. Thus $T \notin v(X, \neg(A \to B))$.
\end{itemize}

\noindent This completes the proof. If $X \nvdash_\mathbb{L} A$, then the canonical model is a counter-model: $X \not\models_{\mathcal{M_C}} A$, \textit{i.e.}, there is a node $Y \in S_C$ such that $T \in v_{\mathcal{M_C}}(Y, B)$ for each $B \in X$ and $T \notin v_{\mathcal{M_C}}(Y, A)$.

\subsection{Concluding remarks on the logic(s)}
Some comments are now in order about the interpretation of our logics of scientific discovery. While they all share the same rules, the minimal \textbf{LScD} has axioms A1-A11 only. Other logics will be justified by their capturing how laboratories share data (and not by logical convenience). For example, while axiom A12 would be logically convenient, it also corresponds to the strong requirement that the $R_1$ accessibility relation be reflexive, $R_1xxx$. Thus for every lab there would be at least one test on which it undertakes all the work: intiating the test, completing the test, collating the results of the tests, and evaluating them.

It is also worth noting that none of the \textbf{LScD} logics have contraposition, not even in rule form. That is,
\begin{center}
$A \rightarrow \neg{}B / B \rightarrow \neg{}A$
\end{center}
is not sound. The following model $\mathcal{M} = \langle \mathscr{F}, v \rangle$ demonstrates this: we set $S = \lbrace u, x, y \rbrace$, $L = \lbrace u \rbrace$, $R_1 = \lbrace \langle u, u, u \rangle, \langle u, x, x \rangle, \langle u, x, y \rangle, \langle u, y, y \rangle, \langle x, x, x \rangle,$ $\langle x, x, y \rangle, \langle x, y, y \rangle, \langle y, y, y \rangle \rbrace$, $R_2 = \lbrace \langle u, u, u \rangle, \langle x, x, x \rangle, \langle y, y, y \rangle, \langle y, x, x \rangle \rbrace$;  $v_u(A) = v_x(A)$ $=  v_y(A) =  \emptyset$; $v_u(B) = v_x(B) = v_y(B) =  \lbrace T, F  \rbrace$. Taking $z \leq w$ to obtain when $(\exists{}u \in L)R_{1}uzw$, we have $\leq{} = \lbrace \langle u,u\rangle,  \langle x, x \rangle, \langle x, y \rangle, \langle y, y\rangle \rbrace$. We then set $\mathscr{F} = \langle S, L, \leq{}, R_1, R_2 \rangle$. This model satisfies all of conditions \textit{i}) -- \textit{x}) in Definition \ref{defFrame}. Moreover, $T \in v_u(A \rightarrow \neg{}B)$, $T \notin v_u(B \rightarrow \neg{}A)$.

We can motivate the failure of contraposition with the hoary example of Eddington's expedition, set up to determine whether its position appears to shift when a star's light passes near a massive object. It's a different matter to test whether a star's position does not appear to shift when its light does not pass near a massive object. Even sillier: if you pet your cat it will purr -- this is easily tested. But checking whether your cat's not purring when you're not petting it is not so simple. The former can be checked from your couch, the latter may require significant mobility and stealth. The absence of contraposition goes to explain the difference between our approach and Routley's, a number of constraints he places on the second accessibility relation being designed to deliver contraposition. Likewise Oshini places constraints on the second accessibility relation with the aim of providing semantics for a co-implication connective. The failure of contraposition in our system is mitigated to some extent by the holding of the rules for negated prefixing and negated suffixing (R5 and R6). It's failure in general, though, makes comparison with extant systems of relevance logic difficult.  


The \textbf{LScD} logics are very flexible: but scientific practice requires even more flexibility. We have operated on the assumption that lab reports are unambiguous, even when reporting their ambiguity. In the following part we remove this assumption by incorporating probabilities into our framework. 
\newpage

\begin{center}
\LARGE

Part II: Probabilities
\end{center}

\section{Introduction}
In Part I we developed Logics of Scientific Discovery that describe certain aspects of scientific practice. In this part we take note of another aspect of scientific practice: lab resorts almost always involve probabilities. We define probabilities -- or, more accurately, appropriately generalized probability functions --  with our Logic(s) of Scientific Discovery, not classical logic, as the underlying logic. We provide relative frequency and betting quotient interpretations.

We begin with probabilities at the level of the individual laboratory, i.e. probabilities for propositions in the $\rightarrow$-free fragment of our logics (Dunn-Belnap logic). Later in this part we give an analysis of the interaction of probabilities over networks of laboratories, so defining probabilities over the full vocabulary. We are led to develop analogues of Bayesian conditionalization, Jeffrey conditionalization and Adams conditioning. Probabilities of conditionals are dealt with in the last section.

\section{Probabilities in the Laboratory}\label{plab}

We begin with a generalized version of the Kolmogorov axioms.

\begin{definition}[Probabilities]\label{Probabilities}
A probability space is a pair $\langle \mathcal{L}, p \rangle$, where $\mathcal{L}$ is the set of all $\to$-free formulas generated by the set $At(\mathcal{L})$ of atomic formulas in $\mathcal{L}$ (see \S{}\ref{LLR}), $\vDash$ is the relation of logical consequence specified in \ref{defConsequence}, and $p$ is a function from $\mathcal{L}$ into the real numbers satisfying:
\begin{enumerate}[label=\roman*)]
\item\label{probax1} for all $A \in \mathcal{L}$, $0 \le p(A) \le 1$,
\item\label{probax2} for all $A, B \in \mathcal{L}$, if $A \vDash B$ then $p(A) \le p(B)$,
\item\label{probax3} for all $A, B \in \mathcal{L}$, $p(A \wedge B) + p(A \vee B) = p(A) + p(B)$,
\item\label{probax4}for all $A, B \in \mathcal{L}$, if $p(B) > 0$ then $p(A|B) = \dfrac{p(A \land B)}{p(B)}$.
\end{enumerate}
\end{definition}

\noindent As they stand, the axioms admit a trivialising interpretation: a function which uniformly assigns the value $0$ to all members of $\mathcal{L}$, leaving $p(A|B)$ undefined for all pairs $A, B$. We could exclude it by adding this principle as a further axiom:
\begin{center}
for some $A \in \mathcal{L}$, $0 < p(A)$.
\end{center}

Axiom \ref{Probabilities} \ref{probax3} is written to account for the non-Boolean structure of the language.\footnote{\textit{Cf.} \cite[pp.\ 107--108]{Priest2006}. For related analyses see, \textit{e.g.}, \cite{Mares1997,Mares2006,Zhou2013}.} As we employ Dunn-Belnap four-valued logic, negation does not determine partitions; as negation and partitions come apart, we are no longer guaranteed that $p(A \wedge \neg A) = 0$. Indeed the usual statement of the additivity axiom for propositions\footnote{By which we mean:
if $\vDash \neg{}(A \land B)$ then $p(A \lor B) = p(A) + p(B)$.} is devoid of application. We replace it with what is in classical probability theory an easily derived consequence.



Suppose we hold it possible that $A$ be both (reported) true and (reported) false. Then we may assign a non-zero probability to $A \land \neg{}A$ and thus, from Axiom \ref{Probabilities} \ref{probax3}, it follows that $p(A) + p(\neg{}A) > p(A \lor \neg{}A)$. Suppose, next, that we hold it possible that $A$ be neither (reported) true nor (reported) false. In close analogy to the previous case, taking the \emph{uncertainty} $u(B)$ assigned a proposition $B$ to be $1 - p(B)$, we should assign a non-zero uncertainty to $A \lor \neg{}A$ and thus, from Axiom \ref{Probabilities} \ref{probax3}, we find  that $u(A) + u(\neg{}A) > u(A \land \neg{}A)$.\footnote{We take this notion of uncertainty from \cite{Adams1975,Edgington1992}.}

Axiom \ref{Probabilities} \ref{probax2}, too, is, in the classical setting, derived from the same additivity axiom and the constraint---{}not sound in our setting!\ but often, classically, adopted as an axiom---{}that $p(A) + p(\neg{}A) = 1$.\footnote{Classically, if $A \vDash_{CL} B$ then $\vDash_{CL} \neg(A \land \neg{}B)$, and hence, from \ref{probax1}, respected classically, and the (classical) additivity axiom, we have that $1 \geq p(A \vee \neg{}B) = p(A) + p(\neg{}B) = p(A) + (1 - p(B))$ whence $p(A) \leq p(B)$.}

From Axiom \ref{Probabilities} \ref{probax4} we find, thanks to the resources of Dunn--Belnap logic, that when $p(B) > 0$, $p(.|B)$ satisfies Definition \ref{Probabilities} \ref{probax1} -- \ref{probax3} with the upper bound in Axiom \ref{Probabilities} \ref{probax1} attained by $p(B|B)$; moreover, when $p(C|B) > 0$, $\dfrac{p(A \wedge C|B)}{p(C|B)} = p(A|B \wedge C)$. 

We now turn to providing relative frequency and betting quotient interpretations of the axioms.


\subsection{Relative Frequencies}\label{RelativeFrequencies}
The orthodox relative frequency interpretation is readily adapted to our framework, the only necessary modification needed being separate definitions of the frequency of a proposition and its negation. With outcomes of the $n$\textsuperscript{th} trial as stipulated, we have:

\begin{definition}[Relative Frequencies\label{freq}]
\begin{equation}
    freq_n(A) = 
    \begin{cases}
      freq_{n-1}(A) + 1 & \text{if}\ T \in v(A), \\
      freq_{n-1}(A) & \text{if}\ T \notin v(A).
    \end{cases}
  \end{equation}

\begin{equation}
    freq_n(\neg{}A) = 
    \begin{cases}
      freq_{n-1}(\neg{}A) & \text{if}\ F \notin v(A), \\
      freq_{n-1}(\neg{}A) + 1 & \text{if}\ F \in v(A).
    \end{cases}
  \end{equation}
\end{definition}

\noindent Relative frequency is $rfreq_n(A) = \frac{freq_n(A)}{n}$. 

\begin{lemma}[\textit{rfreq} is a probability]
rfreq satisfies the axioms of Definition \ref{Probabilities}.
\end{lemma}

\begin{proof}
Axiom \ref{Probabilities} \ref{probax1}: Obvious.\\
Axiom \ref{Probabilities} \ref{probax2}: By Definition \ref{defConsequence}, if $A \vDash B$, then whenever $T \in v(A), T \in v(B)$ and hence by Definition \ref{freq}, for any $n$, $freq_n(B) \ge freq_n(A)$.\\
Axiom \ref{Probabilities} \ref{probax3}: By induction, for any $n$, $freq_n(A \vee B) + freq_n(A \wedge B) = freq_n(A) + freq_n(B)$.\\
Axiom \ref{Probabilities} \ref{probax4}: The conditional probability $p(A|B)$ is the relative frequency of $A$ restricted to trials in which $B$ is the outcome, that is, $\frac{freq_n(A \wedge B)}{freq_n(B)}$, i.e. $\frac{rfreq_n(A \wedge B)}{rfreq_n(B)}$, assuming that $B$ has occurred, i.e., $freq_n(B) > 0$.
\end{proof}


As it stands, this is a finite frequency interpretation. A limiting relative frequency interpretation can easily be constructed.

\subsection{Betting Quotients}

The betting quotient interpretation of probability is also readily adapted to our framework.

\begin{definition}[Bet]
A \emph{bet} on (proposition) $A$ with (positive or negative) \emph{stake} $S$ at \emph{betting quotient} $p$ pays $(1 - p)S$ to the bettor if $A$ takes a designated value ($\lbrace T \rbrace, \lbrace T, F \rbrace$) and pays $pS$ to the bookmaker if it doesn't.
\end{definition}

\begin{definition}[Conditional bet]
A \emph{(conditional) bet} on (proposition) $B$ conditional on (proposition) $A$ with (positive or negative) \emph{stake} $S$ at \emph{betting quotient} $p$ pays nothing to either bettor or bookmaker if $A$ does not take a designated value and otherwise pays $(1 - p)S$ to the bettor if $B$ takes a designated value and pays $pS$ to the bookmaker if it doesn't.
\end{definition}

\noindent Notice that in these definitions the amounts paid to bettor and to bookmaker may be negative; equivalently, stakes are always positive but the roles of bettor and bookmaker

\begin{definition}[Dutch book]
A bettor faces a \em ph{Dutch book} on a family of bets if, given the chosen betting quotients and stakes, she faces certain loss, \textit{i.e.}, on all assignments of sets of truth-values to atomic propositions, the bettor suffers a net loss (which is paid to the bookmaker).
\end{definition}

\begin{thm}[Dutch Book Argument]\label{dutchbookargument}
A bettor may face a Dutch Book on a finite family of bets, through an unfortunate choice of stakes, if her betting quotients do not satisfy Axioms \ref{probax1} -- \ref{probax4}.
\end{thm}
\begin{proof}
\textbf{Axiom \ref{Probabilities} \ref{probax1}.} Firstly, $-pS_{1}$ and $(1 - p)S_{1}$ are both negative if, and only if, either (\textit{i}) $S_{1} < 0$ and $p < 0$ or (\textit{ii}) $S_{1} > 0$ and $p > 1$. ---{} Our bettor faces a Dutch Book on a single bet if, and only if, $p < 0$ or $p > 1$.

Now consider bets on $A$ at betting quotient $p$ with stake $S_{1}$ and at betting quotient $q$ with stake $S_{2}$. From the immediately preceding, the bettor immediately faces a Dutch book if any of these are the case: $p < 0$, $p > 1$, $q < 0$, $q > 1$ so we suppose that $0 \leq p \leq 1$ and that $0 \leq q \leq 1$.

Let
\begin{align*}
G_{1} &= -pS_{1} - qS_{2} & \qquad\qquad\qquad &  T \notin v(A) \\
G_{2} &= (1 - p)S_{1} + (1 - q)S_{2} & \qquad\qquad\qquad &  T \in v(A).
\end{align*}
If $G_{1} < 0$ and $G_{2} < 0$, $(1 - p)G_{1} + pG_{2} < 0$, hence
\[ (p - q)S_{2} < 0.\]
Hence $p \neq q$.

Now, supposing that $p \neq q$, choose $S$ of the same sign as $p - q$ and set $S_{1} = S$, $S_{2} = -S$. We then have
\begin{align*}
G_{1} &= -pS + qS = (q - p)S < 0\\
G_{2} &= (1 - p)S - (1 - q)S = (q - p)S < 0.
\end{align*}
Granted that $0 \leq p \leq 1$ and that $0 \leq q \leq 1$, our bettor faces a Dutch Book on the pair of bets if, and only if, $p \neq q$.

\paragraph{Axiom \ref{Probabilities} \ref{probax2}.} There are two cases to consider. Firstly if, in addition to $A \vDash B$, $B \vDash A$, there are just two possibilities when we consider a pair of bets on $A$ and $B$ at betting quotients $p$ and $q$ and stakes $S_{1}$ and $S_{2}$, respectively: neither $A$ nor $B$ takes a designated value or both do.

From above, the bettor immediately faces a Dutch book if any of these are the case: $p < 0$, $p > 1$, $q < 0$, $q > 1$ so we suppose that $0 \leq p \leq 1$ and that $0 \leq q \leq 1$. Algebraically, the argument now proceeds exactly as above for there are just these two cases to consider:

\begin{align*}
G_{1} &= -pS_{1} - qS_{2} & \qquad\qquad\qquad &  T \notin v(A), T \notin v(B) \\
G_{2} &= (1 - p)S_{1} + (1 - q)S_{2} & \qquad\qquad\qquad &  T \in v(A), T \in v(B).
\end{align*}
Consequently, granted that $0 \leq p \leq 1$ and that $0 \leq q \leq 1$, our bettor faces a Dutch Book on the pair of bets if, and only if, $p \neq q$.

The second case: $A \vDash B$ but $B \nvDash A$. There are three possibilities when we consider a pair of bets on $A$ and $B$ at betting quotients $p$ and $q$ and stakes $S_{1}$ and $S_{2}$, respectively: neither $A$ nor $B$ takes a designated value, $B$ takes a designated value but $A$ does not, both take a designated value. As before, the bettor immediately faces a Dutch book if any of these are the case: $p < 0$, $p > 1$, $q < 0$, $q > 1$ so we suppose that $0 \leq p \leq 1$ and that $0 \leq q \leq 1$.

Let
\begin{align*}
G_{1} &= -pS_{1} - qS_{2} & \qquad\qquad\qquad & T \notin v(A), T \notin v(B) \\
G_{2} &= -pS_{1} + (1 - q)S_{2} & \qquad\qquad\qquad &  T \notin v(A), T \in v(B) \\
G_{3} &= (1 - p)S_{1} + (1 - q)S_{2} & \qquad\qquad\qquad & T \in v(A), T \in v(B) .
\end{align*}
If $G_{1} < 0$, $G_{2} < 0$ and $G_{3} < 0$, $(1 - q)G_{1} + qG_{2} < 0$ and $(1 - p)G_{2} + pG_{3} < 0$, hence
\[ -pS_{1} < 0 \textrm{~and~} (1 - q)S_{2} < 0.\]
From this we see that $S_{1} > 0$ and $S_{2} < 0$. Now, $(1 - q)G_{1} + qG_{2} < 0$, hence $(q - p)S_{1} < 0$. And so $p > q$.

Now, supposing that $p > q$, choose $S_{1} > 0$ and set $S_{2} = -S_{1}$. We find that
\begin{align*}
G_{1} &= -pS_{1} - qS_{2} = (q - p)S_{1} < 0\\
G_{2} &= -pS_{1} + (1 - q)S_{2} = -pS_{1} - (1 - q)S_{1} < 0\\ 
G_{3} &= (1 - p)S_{1} + (1 - q)S_{2} = (q - p)S_{1} < 0.
\end{align*}
Granted that $0 \leq p \leq 1$ and that $0 \leq q \leq 1$, our bettor faces a Dutch Book on the triple of bets if, and only if, $p > q$.

\paragraph{Axiom \ref{Probabilities} \ref{probax3}.} Four bets are to be made: on $A$ at betting quotient $p$ and stake $S_{1}$, on $B$ at betting quotient $q$ and stake $S_{2}$, on $A \wedge B$ at betting quotient $r$ and stake $S_{3}$ and on $A \vee B$ at betting quotient $s$ and stake $S_{4}$. As $A \wedge B \vDash A$, $A \wedge B \vDash B$, $A \vDash A \vee B$ and $B \vDash A \vee B$, from above the bettor immediately faces a Dutch book if any of these are the case: $p < 0$, $p > 1$, $q < 0$, $q > 1$, $r < 0$, $r > 1$, $s < 0$, $s > 1$, $r > p$, $r > q$, $p > s$, $q > s$ so we suppose that $0 \leq r \leq p \leq s \leq 1$ and that $0 \leq r \leq q \leq s \leq 1$.

There are two special cases to consider. Firstly, if $A \vDash B$ then $A \vDash A \wedge B$ and $A \vee B \vDash B$, hence, if the the bettor is not immediately to face a Dutch book, $p = r$ and $q = s$, whence $p + q = r + s$. Secondly, if $B \vDash A$ then $B \vDash A \wedge B$ and $A \vee B \vDash A$, hence, if the the bettor is not immediately to face a Dutch book, $p = s$ and $q = r$, whence $p + q = r + s$.

(If $A \vDash B$ and $B \vDash A$ then, if the the bettor is not immediately to face a Dutch book, $p = q =  r = s$, whence $p + q = r + s$.)

We now suppose that $A \nvDash B$ and $B \nvDash A$.

Let
\begin{align*}
G_{1} &= -pS_{1} - qS_{2} - rS_{3} - sS_{4} & \qquad &  T \notin v(A), T \notin v(B) \\
G_{2} &= -pS_{1} +( 1 - q)S_{2}  - rS_{3} + (1 - s)S_{4} & \qquad &  T \notin v(A), T \in v(B) \\
G_{3} &= (1 - p)S_{1} - qS_{2}  - rS_{3} + (1 - s)S_{4} & \qquad &  T \in v(A), T \notin v(B) \\
G_{4} &= (1 - p)S_{1} + (1 - q)S_{2} + (1 - r)S_{3} + (1 - s)S_{4} & \qquad &  T \in v(A), T \in v(B) .
\end{align*}

If $G_{1} < 0$, $G_{2} < 0$, $G_{3} < 0$ and $G_{4} < 0$ then $G_{12} = (1 - s)G_{1} + sG_{2} < 0$,   $G_{13} = (1 - s)G_{2} + sG_{3} < 0$, $G_{14} = (1 - s)G_{1} + sG_{4} < 0$, \textit{i.e.},
\begin{align*}
G_{12} &= -pS_{1} + (s - q)S_{2} - rS_{3} < 0\\
G_{13} &= (s - p)S_{1} - qS_{2} - rS_{3} < 0\\
G_{14} &= (s - p)S_{1} + (s - q)S_{2} + (s - r)S_{3} < 0.
\end{align*}
And so, if $G_{1} < 0$, $G_{2} < 0$, $G_{3} < 0$ and $G_{4} < 0$ then $s \neq 0$, hence $s > p$ or $p > 0$ and so $G_{1213} = (s - p)G_{12} + pG_{13} < 0$ and $G_{1214} = (s - p)G_{12} + pG_{14} < 0$, \textit{i.e.},
\begin{align*}
G_{1213} &= s[(s - p - q)S_{2} - rS_{3}] < 0\\
G_{1214} &= s[(s - q)S_{2} + (p - r)S_{3}] < 0.
\end{align*}
If $G_{1} < 0$, $G_{2} < 0$, $G_{3} < 0$ and $G_{4} < 0$ then there are two cases to consider. Firstly, $p > r$ or $r > 0$ in which case $(p - r)[(s - p - q)S_{2} - rS_{3}] + r[(s - q)S_{2} + (p - r)S_{3}] < 0$, \textit{i.e.},
\[ p[(r + s) - (p + q) ]S_{2} < 0. \]
Secondly, $p = r = 0$ in which case $(s - q)S_{2} < 0$.  Either way, $p + q \neq r + s$.

Now, supposing that $p + q \neq r + s$, we choose $S$ to be of the same sign as $(p + q) - (r + s)$ and set $S_{1} = S_{2} = -S_{3} = -S_{4} = S$. We find that
\begin{align*}
G_{1} &= -pS - qS + rS + sS = -[(p + q) - (r + s)]S < 0\\
G_{2} &= -pS + (1 - q)S  + rS - (1 - s)S = -[(p + q) - (r + s)]S < 0\\ 
G_{3} &= (1 - p)S - qS  + rS - (1 - s)S = -[(p + q) - (r + s)]S < 0\\
G_{4} &= (1 - p)S + (1 - q)S - (1 - r)S - (1 - s)S = -[(p + q) - (r + s)]S < 0.
\end{align*}
Granted that $0 \leq r \leq p \leq s \leq 1$ and that $0 \leq r \leq q \leq s \leq 1$, our bettor faces a Dutch Book on the family of four bets  if, and only if, $p + q \neq r + s$.

\paragraph{Axiom \ref{Probabilities} \ref{probax4}.} Three bets are to be made: on $B$ at betting quotient $p$ and stake $S_{1}$, on $A \wedge B$ at betting quotient $q$ and stake $S_{2}$, and on $A$ conditional on $B$ at betting quotient $r$ and stake $S_{3}$. As $A \wedge B \vDash B$, from above the bettor immediately faces a Dutch book if any of these are the case: $p < 0$, $p > 1$, $q < 0$, $q > 1$, $q > p$. We suppose that $0 \leq q \leq p \leq 1$.

Let
\begin{align*}
G_{1} &= -pS_{1} - qS_{2} & \qquad &  T \notin v(B) \\
G_{2} &=  (1-p)S_1 - qS_2 - rS_3& \qquad &  T \notin v(A), T \in v(B) \\
G_{3} &= (1 - p)S_{1} + (1 - q)S_{2} + (1 - r)S_{3} & \qquad & T \in v(A), T \in v(B).
\end{align*}

If $G_{2} < 0$ and $G_{3} < 0$ then $G_{23} = (1 - r)G_{2} + rG_{3} < 0$, \textit{i.e.}, $(1 - p)S_{1} + (r - q)S_{2} < 0$. If, in addition, $G_{1} < 0$, then $pG_{23} + (1 - p)G_{1} < 0$, \textit{i.e.}, $(pr - q)S_{2} < 0$. And so $pr \neq q$.

Now, supposing that $pr \neq q$, choose $S$ to be of the same sign as $pr - q$, set $S_{1} = rS$, $S_{2} = -S$, $S_{3} = S$ and we find that
\begin{align*}
G_{1} &= -prS + qS = -(pr - q)S < 0\\
G_{2} &= (1 - p)rS + qS - rS = -(pr - q)S < 0 \\
G_{3} &= (1 - p)rS - (1 - q)S + (1 - r)S = -(pr - q)S < 0.
\end{align*}

Granted that $0 \leq q \leq p \leq 1$, our bettor faces a Dutch Book on the triple of bets if, and only if, $pr \neq q$.
\end{proof}

\begin{thm}[Converse Dutch Book Argument]\label{conversedba}
A bettor cannot, through an unfortunate choice of stakes, face a Dutch Book on a finite family of bets if her betting quotients satisfy Axioms \ref{Probabilities} \ref{probax1} -- \ref{probax4}.
\end{thm}
\begin{proof} 
Given the language $\mathcal{L}$, let the betting quotients $q(A), A \in \mathcal{L}$ satisfy Axioms \ref{Probabilities} \ref{probax1} -- \ref{probax4}.

A classical probability distribution satisfies these axioms (\textit{cf.} \cite{Paris2001}):
\begin{enumerate}[label=\roman*)]
\item $p$ is a real-valued function such that for all $A \in \mathcal{L}, 0 \le p(A) \le 1$,
\item for all $A, B \in \mathcal{L}$, if $A \vDash B$ then $p(A) \le p(B)$,
\item for all $A, B \in \mathcal{L}, p(A \wedge B) + p(A \vee B) = p(A) + p(B)$,
\item for all $A, B \in \mathcal{L}$, if $p(B) > 0$ then $p(A|B) = \dfrac{p(A \land B)}{p(B)}$,
\end{enumerate}
\begin{enumerate}[label=c\roman*)]
\setcounter{enumi}{4}
\item\label{probaxc1} for all $A \in \mathcal{L}$, if $\vDash A$ then $P(A) = 1$,
\item\label{probaxc2} for all $A \in \mathcal{L}$, if $\vDash \neg{}A$ then $P(A) = 0$.
\end{enumerate}
Here `$\vDash$' stands for classical consequence. In fact in the application we are about to make of this, we can happily strengthen it to mean classical consequence given the semantic account of Dunn--Belnap logic in \S{}\ref{LLR}. We'll indicate this by `$\vDash_{ST}$'. (As is common practice, our meta-language is classical.)

Define a function $P$ on the algebra generated by the (classical) meta\-{}linguistic propositions $T \in v(A), A \in \mathcal{L}$ by setting $P(T \in v(A)) = q(A)$. As follows from the axioms above, $P(T \notin v(A)) = 1 - q(A)$ for all $A \in \mathcal{L}$.

We need to show that in making this assignment there is no conflict between the axioms governing $P$ and the axioms governing $q$. This we do as follows:
\begin{enumerate}[label=\roman*)]
\item As $0  \leq q(A) \leq 1$, $0  \leq P(T \in v(A)) \leq 1$.
\item For all $A, B \in \mathcal{L}$, $T \in v(A) \vDash_{ST} T \in v(B)$ iff, for all valuations $v$, $T \in v(B)$ if $T \in v(A)$ iff $A \vDash B$.
\item For all $A, B \in \mathcal{L}$, $P(T \in v(A) ~\textrm{and}~ T \in v(B)) + P(T \in v(A) ~\textrm{or}~ T \in v(B)) = P(T \in v(A \wedge B)) + P(T \in v(A \vee B)) = q(A \wedge B) + q(A \vee B) = q(A) + q(B) = P(T \in v(A)) + P(T \in v(B))$.
\item For all $A, B \in \mathcal{L}$, if $P(T \in v(B)) > 0$, equivalently, if $q(B) > 0$, then $P(T \in v(A)|T \in v(B)) = \dfrac{P(T \in v(A) ~\textrm{and}~ T \in v(B))}{P(T \in v(B))} = \dfrac{P(T \in v(A \wedge B))}{P(T \in v(B))} = \dfrac{q(A \wedge B)}{q(B)} = q(A|B)$.
\end{enumerate}
Notice too that, say, $P(T \in v(A) ~\textrm{and}~ T \notin v(B)) = P(T \in v(A)) - P(T \in v(A) ~\textrm{and}~ T \in v(B)) = q(A) - q(A \wedge B)$ and $P(T \notin v(A) ~\textrm{and}~ T \notin v(B)) = 1 - P(T \in v(A) ~\textrm{or}~ T \in v(B)) = 1 - q(A \vee B)$. Just as in classical logic, Dunn--Belnap logic has DeMorgan's Laws, Laws of distributivity of `$\land$' over `$\lor$' and \textit{vice versa}, and Double Negation equivalence, so any formula can be expressed in disjunctive normal form as a disjunction of conjunctions of literals. What we rely on here is the pseudo-classical behaviour of $\wedge$ and $\vee$ which arises from their satisfying what, above, we called truisms about truth.


The expected value, relative to $P$, of a bet on $A$ at betting quotient $p$ with stake $S$ is
\[ -P(T \notin v(A))pS + P(T \in v(A))(1 - p)S  = \left( P(T \in v(A)) - p \right) S. \]
Clearly, this is zero if, and only if, $P(T \in v(A)) = p$. And so, when $P$ is set up as above, \textit{i.e.} by setting $P(T \in v(A)) = q(A)$, the expectation is zero for a bet on $A$ at betting quotient $q(A)$, no matter the size and sign of the stake.

The expected value, relative to $P$, of a (conditional) bet on $A$ conditional on $B$ at betting quotient $p$ with stake $S$ is\\

$ P(T \notin v(B))\cdot 0 - P(T \notin v(A)$ $T \in v(B)) pS$ \\

~~~~~~~~~ $+~~ P(T \in v(A)$ and $T \in v(B))(1 - p)S$ \\

~~~~~~~~~ $= [ P(T \in v(A)$ and $T \in v(B)) - pP(T \in v(B)) ]S.$\\

And clearly this is zero just if $p \times P(T \in v(B)) = P(T \in v(A) ~\textrm{and}~ T \in v(B))$, that is, just if $p \times P(T \in v(B)) = P(T \in v(A \wedge B))$. And so, when $P$ is set up by setting $P(T \in v(A)) = q(A)$, and so on for the members of $\mathcal{L}$, the expectation is zero for a (conditional) bet on $A$ conditional on $B$ at betting quotient $q(A|B)$, no matter the size and sign of the stake.

Consider a family of $n$ bets on the propositions $A_{1}$, $A_{2}$, \ldots, $A_{n}$ at betting quotients $q(A_{1})$, $q(A_{2})$, \ldots, $q(A_{n})$ and stakes $S_{1}$, $S_{2}$, \ldots $S_{n}$, respectively. Given any assignment of truth-values---{}$\emptyset$, $\lbrace T \rbrace$, $\lbrace F \rbrace$, $\lbrace T, F \rbrace$ ---{} to literals, we can work out the gain/loss on each bet. The net gain/loss is the sum of the gains/losses on the $n$ bets. Consequently we can work out the expected net gain/loss relative to $P$. But classically the expected value of a sum is the sum of the expected values of the summands and, as we have seen, for each bet this is zero (including conditional bets). The expected value is negative if each possible value is negative, as the net gains/losses would be if sure loss was faced. Hence sure loss is not faced: the bettor does not face a Dutch book.
\end{proof}

\subsubsection{Reversed bets}
\begin{definition}[Reversed bet]
A \emph{reversed bet} on (proposition) $A$ with (positive or negative) \emph{stake} $S$ at \emph{betting quotient} $p$ pays $pS$ to the bookmaker if $A$ takes a designated value and pays $(1 - p)S$ to the bettor if it doesn't.
\end{definition}

\begin{definition}[Reversed Conditional bet]
A \emph{reversed (conditional) bet} on (proposition) $B$ conditional on (proposition) $A$ with (positive or negative) \emph{stake} $S$ at \emph{betting quotient} $p$ pays nothing to either bettor or bookmaker if $A$ does not take a designated value and otherwise pays $pS$ to the bookmaker if $B$ takes a designated value and pays $(1 - p)S$ to the bettor if it doesn't.
\end{definition}

Letting $u(A)$ stand for the betting quotient for a reversed bet, we can run analogues of the Dutch book arguments above, switching $-x$ and $1 - x$, for $x = p, q, r, s$, in the characterization of pay-offs, to find that
\begin{thm}[Dutch Book Argument for reversed bets]
A bettor may face a Dutch Book on a finite family of reversed bets, through an unfortunate choice of stakes, if her betting quotients do not satisfy these axioms:
\begin{enumerate}[label=\roman*)]
\setcounter{enumi}{4}
\item\label{probax12} $u$ is a real-valued function such that $\mbox{for all}~A \in \mathcal{L}, 0 \le u(A) \le 1$,
\item\label{probax22} $\mbox{for all}~A, B \in \mathcal{L}$, if $A \vDash B$ then $u(B) \le u(A)$,
\item\label{probax32} $\mbox{for all}~A, B \in \mathcal{L}, u(A \wedge B) + u(A \vee B) = u(A) + u(B)$,
\item\label{probax42} $\mbox{for all}~A, B \in \mathcal{L}$, if $u(B) < 1$ then $u(A|B) = 1 - \dfrac{1 - u(A \land B)}{1 - u(B)}$.
\end{enumerate}
\end{thm}

From the part of the proof of the Dutch Book Argument concerning Axiom \ref{Probabilities} \ref{probax1}, we see that in order to avoid a Dutch book on a bet on $A$ at betting quotient $p$, $0 \leq p \leq 1$, and a reversed bet on $A$ at betting quotient $q$, $0 \leq q \leq 1$, we must set $q = 1 - p$. 

Combining bets and reversed bets, including conditional bets, we have this Dutch Book Argument:
\begin{cor}[Combined Dutch Book Argument]\label{combineddba}
A bettor may face a Dutch Book on a finite family of bets and reversed bets, through an unfortunate choice of stakes, if her betting quotients do not satisfy Axioms \ref{Probabilities} \ref{probax1} -- \ref{probax42} and this further axiom:
\begin{enumerate}[label=\roman*)]
\setcounter{enumi}{8}
\item\label{probaxpu} $\mbox{for all}~A \in \mathcal{L}$, $u(A) = 1 - p(A)$.
\end{enumerate}
\end{cor}

This shows that the betting quotients for reversed bets stand to \S{}\ref{plab}'s uncertainties as the betting quotients for bets stand to probabilities. As that association might make one suspect, Axiom \ref{probaxpu} shows that there is really no need to introduce reversed bets in addition to ordinary bets. (Notice that the argument for Corollary \ref{combineddba} applies in the classical case as well.)

\subsection{Theorem of Total Probability}\label{SectionTTP}

The definition of conditional probability used above is the same as the classical one but, as the setting has changed, we must devote some attention to how it is to be employed in adaptations of classical updating rules. Classical updating rules depend on the Theorem of Total Probability, to which we now turn.

In the classical case we prove the Theorem by relying on the equivalence $A = A \wedge \bigvee\limits_i B_i$, where the $B_i$ form a partition, i.e., $\bigvee\limits_i B_i= \top$ and for $i \neq j, B_i \wedge B_j = \bot$. In the classical case partitions exist as a matter of logic. In the present setting there is no such guarantee. When we have the effect of one, relative to the probability distribution in play (see below), we obtain an analogous Theorem: 

\begin{thm}[Theorem of Total Probability]\label{TTP}
$p(A) = \sum\limits_i p(A|B_i)p(B_i)$ when $p(B_i \wedge B_j) = 0$, $i \neq j$, and $p(\bigvee\limits_i B_i) = 1$.
\end{thm}
We first state and prove 
\begin{lemma}
$p(\bigvee\limits^n_{i = 1} C_i) = \sum\limits^n_{i = 1} p(C_i)$ when $p(C_i \wedge C_j) = 0$, $1 \leq i < j \leq n$.
\end{lemma}
\begin{proof}
Trivially true for $n = 1$.

Suppose the lemma holds for $n = k$. Then, by axioms \ref{probax2} and \ref{probax3}  in Definition \ref{Probabilities}, $p(\bigvee\limits^{k + 1}_{i = 1} C_i) = p((\bigvee\limits^k_{i = 1} C_i) \vee C_{k + 1}) = p(\bigvee\limits^k_{i = 1} C_i) + p(C_{k + 1}) - p((\bigvee\limits^k_{i = 1} C_i) \wedge C_{k + 1})$. By the induction hypothesis, this is $\sum\limits^{k + 1}_{i = 1} p(C_i) - p(\bigvee\limits^{k}_{i = 1} (C_i \wedge C_{k + 1}))$. Since, by axioms \ref{probax1}--\ref{probax3}, $0 \leq p((C_i \wedge C_{k + 1}) \wedge (C_j \wedge C_{k + 1})) \leq p(C_i \wedge C_j) \leq 0$, by the induction hypothesis again, $p(\bigvee\limits^{k}_{i = 1} (C_i \wedge C_{k + 1})) = \sum\limits^k_{i = 1} p(C_i \wedge C_{k + 1})$; by hypothesis, $p(C_i \wedge C_{k + 1}) = 0$, $1 \leq i \leq k$, hence $p(\bigvee\limits^{k + 1}_{i = 1} C_i) = \sum\limits^{k + 1}_{i = 1} p(C_i)$.

The result now follows by induction.
\end{proof}

\begin{proof}[Proof of Theorem of Total Probability]
By axioms \ref{probax1}--\ref{probax3} in Definition \ref{Probabilities}, $p(A \wedge \bigvee_i B_i) = p(A)$, since $1 = p(\bigvee_i B_i) \leq p(A \vee \bigvee_i B_i) \leq 1$. By axiom \ref{probax2}, $p(A \wedge \bigvee_i B_i) = p(\bigvee_i (A \wedge  B_i))$ and since, by axioms \ref{probax1}--\ref{probax3} , $0 \leq p((A \wedge B_i) \wedge (A \wedge B_j)) \leq p(B_i \wedge B_j)$, the result now follows by the preceding lemma and axiom \ref{probax4} in Definition \ref{Probabilities}.
\end{proof}

This Theorem of Total Probability differs from the classical case only in that the probability distribution determines the applicability of the theorem, and, in particular, determines which sets behave enough like partitions. That said, the mechanics of the proof are almost identical to those in the classical case.

\begin{obs}
That a set of propositions behaves like a partition under one probability distribution may well entail that it does so under a related distribution. For example, if $\lbrace B_i \colon 1 \leq i \leq n \rbrace$ \emph{behaves like a partition} under the probability distribution $p$, \textit{i.e.}, $p(B_i \wedge B_j) = 0$, $i \neq j$, and $p(\bigvee_i B_i) = 1$, then, for any $C$ such that $p(C) > 0$, $\lbrace B_i \wedge C\colon 1 \leq i \leq n \rbrace$ \emph{behaves like a partition} under the probability distribution $p(\cdot|C)$. \label{Obscond}
\end{obs}

\subsection{A Diachronic Dutch Book argument}

\begin{definition}[Bayesian Conditionalization] A probability distribution $p$ is updated to the distribution $p^*$ by \emph{Bayesian conditionalization} on $B$ if $p(B) > 0$ and\label{Bayescondsimp}
\begin{center}
for all propositions $A$, $p^*(A) = p(A|B)$.
\end{center} 
\end{definition}
As pointed out by van Fraassen \cite{van1989laws}, diachronic Dutch Book arguments require the assumption of an announced updating strategy (and hence vulnerability can be avoided by not announcing such a strategy). In our framework, we could simple take having previously announced---{}or conventionally instituted---{}updating strategies to be a feature of well-behaved labs. If we do, then there is a Dutch Book argument for adhering to the strategy.

\begin{thm}[Diachronic Dutch Book for Bayesian Conditionalization]\label{ddba}
A bettor may face a Dutch Book on a finite family of bets, through an unfortunate choice of stakes, if, having announced an updating strategy, her betting quotients do not satisfy Definition \ref{Bayescondsimp}.
\end{thm}

\begin{proof}
Three bets are to be made: on $B$ at betting quotient $p$ and stake $S_1$, on $A \wedge B$ at betting quotient $q$ and stake $S_2$ and on $A$ updated on $B$ at betting quotient $r$ and stake $S_3$. Assume $p > q$. Let

\begin{align*}
    G_1 & = pS_1 & T \notin v(B)\\
    G_2 & = (1-p)S_{1} - qS_{2} - rS_{3} & T \notin v(A), T \in v(B)\\
    G_3 & = (1-p_{1})S_{1} + (1-q)S_{2} + (1-r)S_{3} & T \in v(A), T \in v(B).
\end{align*}

The argument concerning Axiom \ref{probax4} in the proof of Theorem \ref{dutchbookargument} now applies, given some trivial modifications, as it does to the case $q < p$.
\end{proof}

The following converse Diachronic Dutch Book argument guarantees the consistency of the strategy:

\begin{thm}[Converse Dutch Book Argument for Conditionalization]
Having announced the she will follow the updating stratgey of Bayesian conditonalization (Definition \ref{Bayescondsimp}), a bettor cannot, through an unfortunate choice of stakes, face a Dutch Book on a finite family of (diachronic) bets.
\end{thm}

\begin{proof}
The proof is essentially that of  \cite{Skyrms1987}.
Given that the bettor conditionalizes, we can translate her bets at different times to bets at one time. For any bet on a proposition $A$ offered at a later time we substitute a conditional bet at the earlier time on $A$, where the condition is some truth known at that earlier time. Theorem \ref{conversedba} ensures that the bettor's expectation of loss on this (synchronic) family of bets is 0.
\end{proof}

\subsection{Jeffrey Conditionalization}

We now turn to a more general form of updating, where beliefs over a partition $\{B_i:1 \leq i \leq n\}$ change exogenously, using Howson and Urbach's \cite{HowsonandUrbach} felicitous term, while degrees of belief conditional on the members of the partition remain the same. When this happens the probabilities over other propositions need to be redistributed, that is, we need to go from a probability function $p$ over propositions to a new probability function $p^*$. In the classical case we use the Theorem of Total Probability with the given equalities $p(A|B_i)$ = $p^*(A|B_i)$ to obtain \[p^*(A) = \sum\limits_i p(A|B_i)p^*(B_i),\] \textit{i.e.}, Jeffrey conditionalization. 

We can use Theorem \ref{TTP} to obtain an analogous form:

\begin{definition}[Jeffrey conditionalization]
\[p^*(A) = \sum\limits_i p(A|B_i)p^*(B_i),\]
when $\sum \limits_{i=1}^n p^*(B_i) = 1$ and $p^*(B_i \wedge B_j) = 0$ for $i \neq j$.
\end{definition}

\noindent Notice that, by stipulation, the $B_i$'s behave as a partition with respect to the distribution $p^*$ \emph{but not necessarily the distribution $p$}. It's not where you're coming from, it's where you're going to that matters.

\subsubsection{Dutch Book arguments regarding Jeffrey conditionalization}

Even a brief perusal of the Dutch Book arguments for probability kinematics provided by Brad Armendt \cite{Armendt1980} and Bryan Skyrms \cite{Skyrms1987} shows them to be far too long to replicate, let alone adapt to the current setting, here. We may return to this topic on another occasion.

\subsection{Contrasts and comparisons with other approaches} 

As far as we know, we are the first to have introduced a relative frequency interpretation in a four-valued framework.

Edwin Mares \cite{Mares2014} offers a Dutch Book argument for a set of axioms intended to be applied to a broad class of structures. His approach is semantic: where we assign probabilities to sentences of the language, Mares goes \textit{via} models in which the sentences are interpreted. Furthermore, we have no analogue of his axiom
\begin{center}
  If $W \in \mathcal{A}$ then $P(W) = 1$ and if $\emptyset \in \mathcal{A}$ then $P(\emptyset) = 0$
\end{center}  
where $\mathcal{A}$ is an algebra of subsets of $W$. There is also a stylistic difference in the way we set out the Dutch Book argument. Mares works in terms of expected values;\footnote{One might wonder what exactly an expected value is when taken relative to a function that, \textit{ex hypothesi}, does not satisfy the probability axioms.} we provide arguments in the style of de Finetti and only introduce expectations in our Converse Dutch Book arguments. We offer these for the synchronic and diachronic cases; Mares offers neither.

J. Michael Dunn \cite{Dunn2010} defines probabilities within a four-valued framework where the probabilities are classical. Dunn does not offer a frequency interpretation, although \S{}\ref{RelativeFrequencies}  could be rewritten in his terms. Dunn's approach takes the four values to be independent, so that $p(v(A)=T) + p(v(A)=F) + p(v(A)=\{T,F\}) + p(v(A)=\emptyset)= 1.$ The primary distinction for us, in the case of both bets and frequencies is whether a proposition takes a designated value or not: we have a two-way split, while Dunn has a four-way split.  Dunn's approach is more granular, capturing distinctions between sequences of values such as $\langle\emptyset,\{T,F\}\rangle$ and $\langle F, T \rangle$.\footnote{This was pointed out to us by Mike Dunn.} Our approach ties probability more closely to the entailment relation of Dunn--Belnap logic.

\section{Co\"{o}rdinated Updates and Conditionals}

It is now time to return to the full language of the first part. So far we have equipped each laboratory with a probability function. We now index those probabilities by laboratory and turn to how they may co\"{o}rdinate their results. Indexing the probability functions leads to a modal version of probability, allowing us to examine analogues of Bayesian conditionalization, Jeffery conditionalization, and Adams' conditioning (as Richard Bradley calls it \cite{Bradley2005}); we also say a little on the probability of conditionals.

\subsection{Co\"{o}rdinated Conditional Probability}

Returning to our fundamental motivation: we have a laboratory $x$ that collects and evaluates data obtained from a laboratory $y$ that initiates a test set-up and which is then followed up by a laboratory $z$. We will now focus on labs that are in regular close contact: labs that hold meetings at designated times to coördinate their research. During these meetings, the labs may choose to follow the advice of the others by adopting their probabilities in their domain of expertise, provided doing so is coherent. We shall confine attention to the updates of the lab $x$ which has oversight.

Co\"{o}ordination might look like this: lab $x$ begins with some prior probability $p_x(A)$. $x$ turns to lab $y$'s expertise for a relevant likelihood $p_y(B|A)$, and to $z$'s expertise for $p_z(B)$. The net effect of this updating process is given by

\begin{definition}[Co\"{o}rdinated  Conditionalization]\label{diacondit}
Having obtained values for $p_y(B|A)$ and $p_z(B)$, $p_x(A)$ is updated by \emph{co\"{o}rdinated conditionalization on $B$} to $p^*_x(A)$, where
\[ p^*_x(A) = \frac{p_y(B|A)p_x(A)}{p_z(B)}. \]
\end{definition}

We can think of Definition \ref{diacondit} as a consistency constraint on group activity. We should also note that there may be many pairs of labs $y$ and $z$ reporting to $x$: we take $x$ as updating piecemeal as labs report their results.\footnote{We could adopt a different updating strategy. We could wait until all information is in, and then update by aggregating laboratory probabilities \textit{via} opinion pooling. Indeed, laboratory $x$ could lay out a protocol for classical statistical methods.}

\begin{thm}\label{CCGeneral}
When $p_x(B) \cdot p_y(B) \cdot p_z(B) > 0$, $p^*_x(\cdot)$ is obtained by co\"{o}rdinated conditionalization on $B$ if, and only if,
\begin{enumerate}[label=\roman*)]
    \item $p^*_x(B) = \dfrac{p_x(B)}{p_z(B)}$;
    \item for all propositions $A$, $p_x^*(A) = 0$ if, and only if, $p_x(A) = 0$ or $p_y(B|A) = 0$;
    \item for all propositions $A$ and $C$, if $p_x^*(A) \cdot p_x^*(C) > 0$ then $\dfrac{p_x(A)}{p_x^*(A)} \cdot p_y(B|A) = \dfrac{p_x(C)}{p_x^*(C)} \cdot p_y(B|C)$.
\end{enumerate}
\end{thm}
\begin{proof}
If $p^*_x(\cdot)$ is obtained by co\"{o}rdinated conditioning on $B$ then, since  $p_y(B) > 0$, $P_y(B|B) = 1$ and i) -- iii) obviously obtain.

Conversely, if i) -- iii) obtain then, if $p_x(A) = 0$ or $p_y(B|A) = 0$, $p_x^*(A) = 0$, so suppose that neither obtains.  Then, since $\dfrac{p_x(B)}{p_z(B)} \neq 0$, $p_x^*(B) \neq 0$ and $\dfrac{p_x(A)}{p_x^*(A)} \cdot p_y(B|A) = \dfrac{p_x(B)}{p_x^*(B)} \cdot p_y(B|B) = \dfrac{p_x(B)}{p_x^*(B)} = p_z(B)$, hence $p^*_x(A) = \dfrac{p_y(B|A)p_x(A)}{p_z(B)}$.
\end{proof}

\noindent Theorem \ref{CCGeneral} characterizes a particular form of promiscuous adoption of others' opinions motivated by our interpretation. Of particular interest are the following special cases.

\subsection{Co\"{o}rdinated Bayesian Conditionalization}
\begin{definition}[Co\"{o}rdinated Bayesian Conditionalization]  
Let $\lbrace B_1, B_2 \rbrace$ behave as a partition under the probability distribution $p_z$, \textit{i.e.} $p_z(B_1 \wedge B_2) = 0$ and $p_z(B_1 \vee B_2) = 1$, and suppose that $p_y(B_1) > 0$ and that $p_y(B_1 \wedge B_2) = p_z(B_1 \wedge B_2)$. Then  $p^*_x$ is said to be obtained from $p_y$ and $p_z$ by \emph{co\"{o}rdinated Bayesian conditionalization} on $B_1$ just in case, for all propositions $A$, \label{Bayescond2}
\begin{center}
    $p^*_x(A) = p_y(A|B_1).$
\end{center}
\end{definition}

\begin{thm}
When $\lbrace B_1, B_2 \rbrace$ behaves as a partition under the probability distribution $p_z$ and $p_y(B_1) > 0$ and $p_y(B_1 \wedge B_2) = p_z(B_1 \wedge B_2)$ then $p^*_x(.)$ is obtained from $p_y$ by co\"{o}rdinated Bayesian conditionalization on $B_1$ if, and only if,
\begin{enumerate}[label=\roman*)]
    \item Extremal: $p_x^*(B_1) = 1$ and $p_x^*(B_2) = 0$;
    \item $x$-$y$ rigidity: for all propositions $A$, $p^*_x(A|B_1) = p_y(A|B_1)$;
    \item Partition: $\lbrace B_1, B_2 \rbrace$ behaves as a partition under $p_x^*$.
\end{enumerate}
\end{thm}
\begin{proof}
Suppose that $\lbrace B_1, B_2 \rbrace$ behaves as a partition under $p_z$, that $p_y(B_1) > 0$, that $p_y(B_1 \wedge B_2) = p_z(B_1 \wedge B_2)$, and that $p^*_x(.)$ is obtained from $p_y$ and $P_z$ by co\"{o}rdinated Bayesian conditionalization on $B_1$.\\
\

As $p_y(B_1) > 0$, $p_y(B_1|B_1)$ and $p_y(B_2|B_1)$ are well defined; moreover, $p_x^*(B_1) = p_y(B_1|B_1) = \dfrac{p_y(B_1 \wedge B_1)}{p_y(B_1)} = \dfrac{p_y(B_1)}{p_y(B_1)} = 1$ and $p_x^*(B_2) = p_y(B_2|B_1) = \dfrac{p_y(B_1 \wedge B_2)}{p_y(B_1)} = \dfrac{p_z(B_1 \wedge B_2)}{p_y(B_1)} =  \dfrac{0}{p_y(B_1)} = 0$.\\
\

As $0 \leq p_x^*(B_1 \wedge B_2) \leq p_x^*(B_2) = 0$ and $1 = p_x^*(B_1) \leq p_x^*(B_1 \vee B_2) \leq 1$, $\lbrace B_1, B_2 \rbrace$ behaves as a partition under $p_x^*$.\\
\

$p_x^*(A|B_1) = \dfrac{p_x^*(A \wedge B_1)}{P_x^*(B_1)} = p_x^*(A \wedge B_1) = p_y(A \wedge B_1|B_1) = p_y(A|B_1)$.

Conversely, suppose that i), ii) and iii) obtain. By the Theorem of Total Probability (Theorem \ref{TTP})
\[p_x^*(A) = p_x^*(A|B_1)p_x^*(B_1) + p_x^*(A|B_2)p_x^*(B_2) = p_x^*(A|B_1) = p_y(A|B_1).\]
\end{proof}

\subsection{Co\"{o}rdinated Jeffrey conditionalization}

\begin{definition}[Co\"{o}rdinated Jeffrey conditionalization] When $\lbrace B_i : 1 \leq i \leq n \rbrace$ behaves as a partition under the probability distribution $p_z^*$, \textit{i.e.} $p_z^*(B_i \wedge B_j) = 0$, $i \neq j$, and $p_z^*(\bigvee\limits_i B_i) = 1$, $p_y(B_i) > 0$, $1 \leq i \leq n$, and $p_y(B_i \wedge B_j) = p_z^*(B_i \wedge B_j)$, $i \neq j$, then $p_x^*(\cdot)$ is obtained from $p_y$ and $p_z^*$ by \emph{co\"{o}rdinated Jeffrey conditionalization} on $\lbrace B_i : 1 \leq i \leq n \rbrace$ just in case, for all propositions $A$,
\[ p_x^*(A) = \sum\limits_i p_y(A|B_i)p^*_z(B_i). \]
\end{definition}

\begin{thm}
When $\lbrace B_i : 1 \leq i \leq n \rbrace$ behaves as a partition under the probability distribution $p_z^*$, $p_y(B_i) > 0$, $1 \leq i \leq n$, and $p_y(B_i \wedge B_j) = p_z^*(B_i \wedge B_j)$, $i \neq j$, then $p_x^*(\cdot)$ is obtained from $p_y$ and $p_z^*$ by co\"{o}rdinated Jeffrey conditionalization on $\lbrace B_i : 1 \leq i \leq n \rbrace$, if, and only if,
\begin{enumerate}[label=\roman*)]
    \item $x$-$z$ rigidity: $p_x^*(B_i) = p_z^*(B_i)$;
    \item $x$-$y$ rigidity: for all propositions $A$, $p^*_x(A|B_i) = p_y(A|B_i)$;
    \item Partition: the $B_i$'s behave as a partition with respect to the distributions $p^*_x$.
\end{enumerate}
\end{thm}
\begin{proof}
Suppose that $\lbrace B_i : 1 \leq i \leq n \rbrace$ behaves as a partition under the probability distribution $p_z^*$, $p_y(B_i) > 0$, $1 \leq i \leq n$, and that $p_y(B_i \wedge B_j) = p_z^*(B_i \wedge B_j)$, $i \neq j$.

As $p_y(B_i) > 0$, $p_y(B_i|B_i)$ and $p_y(B_j|B_i)$ are well defined for all $i$, $j$, $1 \leq i, j \leq n$. $p_y(B_i|B_i) = \dfrac{p_y(B_i \wedge B_i)}{p_y(B_i)} = \dfrac{p_y(B_i)}{p_y(B_i)} = 1$. $p_y(B_i|B_j) = \dfrac{p_y(B_i \wedge B_j)}{p_y(B_j)} = \dfrac{p_z^*(B_i \wedge B_j)}{p_y(B_j)} =  \dfrac{0}{p_y(B_j)} = 0$, $i \neq j$. Consequently, 
\[ p_x^*(B_i) = \sum\limits_j p_y(B_i|B_j)p^*_z(B_j) = p^*_z(B_i). \]\\
\

When $i \neq j$, $0 \leq p_y(B_i \wedge B_j|B_k) \leq \min\lbrace p_y(B_i |B_k), p_y(B_j|B_k) \rbrace = 0$ since $i \neq k$ or $j \neq k$. Hence 
$p_x^*(B_i \wedge B_j) = \sum\limits_k p_y(B_i \wedge B_j|B_k)p^*_z(B_k) = 0$.

\begin{align*}
p_x^*(\bigvee\limits_i B_i) &= \sum\limits_j p_y(\bigvee\limits_i B_i|B_j)p^*_z(B_j) = \sum\limits_j \dfrac{p_y((\bigvee\limits_i B_i) \wedge B_j)}{p(B_j)}p^*_z(B_j)\\
&= \sum\limits_j \dfrac{p(B_j)}{p(B_j)}p^*_z(B_j) = \sum\limits_j p^*_z(B_j) = 1.
\end{align*}

Thus $\lbrace B_i : 1 \leq i \leq n \rbrace$ behaves as a partition under $p_x^*$.\\
\

$p_x^*(A|B_i) = \dfrac{p_x^*(A \wedge B_i)}{P_x^*(B_i)} = \dfrac{\sum\limits_j p_y(A \wedge B_i|B_j)p^*_z(B_j).}{P_z^*(B_i)} = p_y(A \wedge B_i|B_i) = p_y(A|B_i)$ since $0 \leq p_y(A \wedge B_i \wedge B_j) \leq p_y(B_i \wedge B_j) = p^*_z(B_i \wedge B_j) = 0$ when $i \neq j$.\\
\

Conversely, suppose that i), ii) and iii) obtain. By the Theorem of Total Probability (Theorem \ref{TTP})
\[p_x^*(A) = \sum\limits_j p_x^*(A|B_i)p_x^*(B_i) =  \sum\limits_j p_y(A|B_i)p_z^*(B_i).\]
\end{proof}

\subsection{Co\"{o}rdinated Adams conditioning}
Richard Bradley \cite{Bradley2005} introduces (and names) \emph{Adams conditioning} as a counterpart to Jeffrey conditionalization, where likelihoods change but the probabilities of certain propositions remain the same. That is, we look for a new probability function where, for some salient $A$, the probability of $A$ remains unchanged, \textit{i.e.} $p(A) = p^*(A)$, while the likelihoods change from $p(\cdot|\cdot)$ to $p^*(\cdot|\cdot)$. We adapt this idea to the present setting.

\begin{definition}[Co\"{o}rdinated Adams  conditioning]  Let $\lbrace A_1, A_2 \rbrace$ and $\lbrace B_1, B_2 \rbrace$ behave as partitions under the probability distribution $p_z$, \textit{i.e.} $p_z(A_1 \wedge A_2) = p_z(B_1 \wedge B_2) = 0$ and $p_z(A_1 \vee A_2) = p_z(B_1 \vee B_2) = 1$. Suppose that  $1 >  p_y(B_1|A_1) > 0$ and that lab $y$ is caused to change its conditional probabilities for $B_1$ given $A_1$ from $p_y(B_1|A_1)$ to $p^*_y(B_1|A_1)$ and for $B_2$ given $A_1$ from $p_y(B_2|A_1)$ to $p^*_y(B_2|A_1)$ where $p^*_y(B_1|A_1) + p^*_y(B_2|A_1) = 1$. Then $x$'s new probabilities $p^*_x$ are said to be obtained by \emph{co\"{o}rdinated Adams conditioning} on this change in conditional probabilities, just in case:\label{Adamscond}
\begin{align*}
    p^*_x(C) &= \dfrac{p^*_y(B_1|A_1)}{p_y(B_1|A_1)} \cdot p_z(A_{1}\wedge{}B_{1}\wedge{}C)\\
    &~~ + \dfrac{p^*_y(B_2|A_1)}{p_y(B_2|A_1)}\cdot p_z(A_{1}\wedge{}B_{2}\wedge{}C) + p_z(A_{2}\wedge{}C).
\end{align*}
\end{definition}

\noindent We prove this analogue of Bradley's Theorem 1 \cite[p.\ 352]{Bradley2005}
\begin{thm}
Where $\lbrace A_1, A_2 \rbrace$ and $\lbrace B_1, B_2 \rbrace$ behave as partitions under the probability distribution $p_z$ and $p_z(B_1|A_1) = p_y(B_1|A_1)$ and $p_z(B_2|A_1) = p_y(B_2|A_1)$, $x$'s new probabilities $p^*_x$ are obtained by Adams conditioning on a change in $y$'s conditional probabilities of $B_1$ given $A_1$ and $B_2$ given $A_1$ iff:
\begin{enumerate}[label=\roman*)]
    \item $x$-$z$ Rigidity: $p^*_x(A_1) = p_z(A_1)$ and $p^*_x(A_2) = p_z(A_2)$;
    \item $x$-$y$ Rigidity: $p^*_x(B_1|A_1) = p^*_y(B_1|A_1)$ and $p^*_x(B_2|A_1) = p^*_y(B_2|A_1)$;
    \item $x$-$z$ Rigidity:  for all propositions $C$, $p^*_x(C|A_{1}\wedge{}B_{1}) =  p_z(C|A_{1}\wedge{}B_{1})$, $p^*_x(C|A_{1}\wedge{}B_{2}) =  p_z(C|A_{1}\wedge{}B_{2})$, and $p^*_x(C|A_2) = p_z(C|A_2)$;
    \item Partition: $\lbrace A_{1}\wedge{}B_{1}, A_{1}\wedge{}B_{2}, A_{2} \rbrace$ behaves as a partition under $p^*_x$.
\end{enumerate}
\end{thm}

\begin{proof} (After \cite[p.\ 362]{Bradley2005}.) If $p^*$ is obtained by Adams conditioning on a change in $y$'s conditional probabilities of $B_1$ given $A_1$ and $B_2$ given $A_1$ then by Definition \ref{Adamscond} and the assumption that $p_z(B_1|A_1) = p_y(B_1|A_1)$ and $p_z(B_2|A_1) = p_y(B_2|A_1)$,

\begin{align*}
p^*_x(A_1) &= \dfrac{p^*_y(B_1|A_1)}{p_y(B_1|A_1)} \cdot p_z(A_{1}\wedge{}B_{1}) + \dfrac{p^*_y(B_2|A_1)}{p_y(B_2|A_1)} \cdot p_z(A_{1}\wedge{}B_{2}) + p_z(A_{1}\wedge{}A_2)\\
&= \dfrac{p^*_y(B_1|A_1)}{p_z(B_1|A_1)} \cdot p_z(A_{1}\wedge{}B_{1}) + \dfrac{p^*_y(B_2|A_1)}{p_z(B_2|A_1)}  \cdot p_z(A_{1}\wedge{}B_{2})\\
&= p^*_y(B_1|A_1) \cdot p_z(A_1) + p^*_y(B_2|A_1) \cdot p(A_1)_z \\
&= [p^*_y(B_1|A_1) + p^*_y(B_2|A_1)] \cdot p_z(A_1)\\
&= 1 \cdot p_z(A_1) = p_z(A_1).
\end{align*}

Noticing that $0 \leq p_z(A_{1}\wedge{}B_{1}\wedge{}A_2) \leq p_z(A_{1}\wedge{}A_2) = 0$ and $0 \leq p_z(A_{1}\wedge{}B_{2}\wedge{}A_2) \leq p_z(A_{1}\wedge{}A_2) = 0$, we have
\begin{align*}
p^*_x(A_2) &= \dfrac{p^*_y(B_1|A_1)}{p_y(B_1|A_1)} \cdot p_z(A_{1}\wedge{}B_{1}\wedge{}A_2)\\
&~~ + \dfrac{p^*_y(B_2|A_1)}{p_y(B_2|A_1)} \cdot p_z(A_{1}\wedge{}B_{2}\wedge{}A_2) + p_z(A_2)\\
&= p_z(A_2).
\end{align*}

Noticing in addition that $p_z(A_{1}\wedge{}B_{1}\wedge{}A_{1}\wedge{}B_{1}) = p_z(A_{1}\wedge{}B_{1})$, $p_z(A_{1}\wedge{}B_{1}\wedge{}C\wedge{}A_{1}\wedge{}B_{1}) = p_z(A_{1}\wedge{}B_{1}\wedge{}C)$, $0 \leq p_z(A_{1}\wedge{}B_{2}\wedge{}C\wedge{}A_{1}\wedge{}B_{1}) \leq p(B_{1}\wedge{}B_{2}) = 0$, and $0 \leq p_z(A_{2}\wedge{}C\wedge{}A_{1}\wedge{}B_{1}) \leq p_z(A_{1}\wedge{}A_{2}) = 0$, we have
\begin{align*}
p^*_x(B_1|A_1) &= \dfrac{p^*_x(A_1 \wedge B_1)}{p^*_x(A_1)}\\
   &= \dfrac{p^*_y(B_1|A_1)}{p_y(B_1|A_1)} \times \dfrac{p_z(A_1 \wedge B_1)}{p_z(A_1)}\\
   &= p^*_y(B_1|A_1) \cdot \dfrac{p_z(B_1|A_1)}{p_y(B_1|A_1)}\\
   &= p^*_y(B_1|A_1).
\end{align*}
And, likewise, we may show that $p^*_x(B_2|A_1) = p^*_y(B_2|A_1).$

\begin{align*}
p^*_x(C|A_{1}\wedge{}B_{1}) &= \dfrac{p^*_x(C\wedge{}A_{1}\wedge{}B_{1})}{p^*_x(A_{1}\wedge{}B_{1})}\\
&= \dfrac{[p^*_y(B_{1}|A_{1})/p_y(B_1|A_{1})] \cdot p_z(C\wedge{}A_{1}\wedge{}B_{1})}{[p^*_y(B_1|A_1)/p_y(B_1|A_{1})] \cdot p_z(A_{1}\wedge{}B_{1})}\\
&= p_z(C|A_{1}\wedge{}B_{1}).
\end{align*}
Similarly, $p^*_x(C|A_{1}\wedge{}B_{2}) = p_z(C|A_{1}\wedge{}B_{2})$ and $p^*_x(C|A_{2}) = p_z(C|A_{2})$.

It is an easy consequence of Definition \ref{Probabilities} \ref{probax4} and Theorem \ref{TTP} that $\lbrace A_{1}\wedge{}B_{1}, A_{1}\wedge{}B_{2}, A_{2} \rbrace$ behaves as a partition under $p_z$ when $\lbrace A_1, A_2 \rbrace$ and $\lbrace B_1, B_2 \rbrace$ both behave as partitions under $p_z$.
Consequently, $p^*_x((A_{1}\wedge{}B_{1}) \wedge (A_{1}\wedge{}B_{2})) = p^*_x((A_{1}\wedge{}B_{1}) \wedge A_{2}) = p^*_x((A_{1}\wedge{}B_{2}) \wedge A_{2}) = 0.$
\begin{align*}
    p^*_x((A_{1}&\wedge{}B_{1}) \vee (A_{1}\wedge{}B_{2}) \vee A_{2})\\
    &= p_z(A_{1}\wedge{}B_{1}\wedge{}((A_{1}\wedge{}B_{1}) \vee (A_{1}\wedge{}B_{2}) \vee A_{2})) \cdot \dfrac{p^*_y(B_1|A_1)}{p_y(B_1|A_1)}\\
    &~~ + p_z(A_{1}\wedge{}B_{2}\wedge{}((A_{1}\wedge{}B_{1}) \vee (A_{1}\wedge{}B_{2}) \vee A_{2})) \cdot \dfrac{p^*_y(B_2|A_1)}{p_y(B_2|A_1)}\\
    &~~ + p_z(A_{2}\wedge{}((A_{1}\wedge{}B_{1}) \vee (A_{1}\wedge{}B_{2}) \vee A_{2}))\\
    &= p_z(A_{1}\wedge{}B_{1}) \cdot \dfrac{p^*_y(B_1|A_1)}{p_z(B_1|A_1)} + p_z(A_{1}\wedge{}B_{2}) \cdot \dfrac{p^*_y(B_2|A_1)}{p_z(B_2|A_1)} + p_z(A_{2})\\
    &= p_z(A_1) \cdot p^*_y(B_1|A_1) + p_z(A_1) \cdot p^*_y(B_2|A_1) + p_z(A_2)\\
    &= p_z(A_1)[p^*_y(B_1|A_1) + p^*_y(B_2|A_1)] + p_z(A_2)\\
    &= p_z(A_1) + p_z(A_2) = 1.
\end{align*} So the $x$-$z$ Dependence, $x$-$y$ Rigidity, $x$-$z$ Rigidity and Partition conditions all obtain.

Conversely, suppose now that the $x$-$z$ Rigidity, $x$-$y$ Rigidity, $x$-$z$ Rigidity and Partition conditions hold. Then, starting from the Theorem of Total Probability,
\begin{align*}
    p^*_x(C) &= p^*_x(C|A_{1}\wedge{}B_{1}) \cdot p^*_x(A_{1}\wedge{}B_{1})\\
    &~~+ p^*_x(C|A_{1}\wedge{}B_{2}) \cdot p^*_x(A_{1}\wedge{}B_{2}) + p^*_x(C|A_{2}) \cdot p^*_x(A_{2})\\
    &= p_z(C|A_{1}\wedge{}B_{1}) \cdot p^*_x(A_{1}\wedge{}B_{1})\\
    &~~ + p_z(C|A_{1}\wedge{}B_{2}) \cdot p^*_x(A_{1}\wedge{}B_{2}) + p_z(C|A_{2}) \cdot p_z(A_{2})\\
    &= \dfrac{p_z(C\wedge{}A_{1}\wedge{}B_{1})}{p_z(A_{1}\wedge{}B_{1})} \cdot p^*_x(A_{1}\wedge{}B_{1})\\
    &~~ + \dfrac{p_z(C\wedge{}A_{1}\wedge{}B_{2})}{p_z(A_{1}\wedge{}B_{2})} \cdot p^*_x(A_{1}\wedge{}B_{2}) + p_z(C\wedge{}A_{2})\\
    &= \dfrac{p^*_x(B_1|A_1) \cdot p^*_x(A_1)}{p_z(B_1|A_1) \cdot p_z(A_1)} \cdot p_z(C\wedge{}A_{1}\wedge{}B_{1})\\
    &~~ + \dfrac{p^*_x(B_2|A_1) \cdot p^*_x(A_1)}{p_z(B_2|A_1) \cdot p_z(A_1)} \cdot p_z(C\wedge{}A_{1}\wedge{}B_{2}) + p_z(C\wedge{}A_{2})\\
    &= \dfrac{p^*_y(B_1|A_1) \cdot p_z(A_1)}{p_y(B_1|A_1) \cdot p_z(A_1)} \cdot p_z(C\wedge{}A_{1}\wedge{}B_{1})\\
    &~~ + \dfrac{p^*_y(B_2|A_1) \cdot p_z(A_1)}{p_y(B_2|A_1) \cdot p_z(A_1)} \cdot p_z(C\wedge{}A_{1}\wedge{}B_{2}) + p_z(C\wedge{}A_{2})\\
    &= \dfrac{p^*_y(B_1|A_1)}{p_y(B_1|A_1)} \cdot p_z(C\wedge{}A_{1}\wedge{}B_{1})\\
    &~~ + \dfrac{p^*_y(B_2|A_1)}{p_y(B_2|A_1)} \cdot p_z(C\wedge{}A_{1}\wedge{}B_{2}) + p_z(C\wedge{}A_{2}).
\end{align*}
\end{proof}

We have generalized three forms of updating, Bayesian conditionalization, Jeffrey conditionalization and Adams conditioning to the co\"{o}rdinated setting. While less well known than the other two, Adams conditioning is particularly apt in the present setting as likelihoods may often be the business of another lab.

\subsection{Probabilities of conditionals}
Our conditional is meant to model regularities---and yet discovery of regularities is rare indeed. It behooves us, therefore, to look for a less strict account of discovery. One option would be to use a finite frequency interpretation: 

\begin{definition}[Relative Frequency of Conditionals and Negated Conditionals]\label{freqConditional}\ \\

The frequency of a regularity $A \to B$ for a laboratory \textit{x} is

$|\{\langle y, z \rangle | R_1xyz \mbox{ and if } T\in v(y, A) \mbox{ then } T\in v(z, B) \}|$.\\ 

The frequency of $\neg(A \to B)$ for a laboratory \textit{x} is

$|\{\langle y, z\rangle | R_2xyz \mbox{ and if } T\in v(y_i, A) \mbox{ then } F \in v(z, B) \}|$. 

\end{definition}

\noindent The relative frequency, of course, is the frequency divided by $|\{\langle y, z \rangle| R_{1}xyz\}|$ and $|\{\langle y, z \rangle| R_{2}xyz\}|$ as appropriate.\\

\noindent Another option would be to elaborate a scheme for assigning betting quotients.

Definition \ref{freqConditional} yields a notion of graded regularity. While in certain rare cases a regularity always holds, there are also regularities which hold only for a certain percentage of cases. (For example, a $90\%$-regularity.) Finally, it is obvious that the probability of conditionals and conditional probabilities are different, as $p(A \rightarrow B)$ is defined when $p(A)=0$, while $p(B|A)$ is not.

\end{document}